\documentclass[12pt]{article}

\usepackage{geometry}

\usepackage[sort&compress,numbers]{natbib}
\usepackage{hyperref}
\usepackage{amsthm}
\usepackage{amssymb}
\usepackage{amsmath}
\usepackage{graphicx}
\usepackage{siunitx} 
\usepackage{bm} 
\usepackage{natbib} 
\usepackage{authblk} 
\usepackage{float}
\usepackage[font=scriptsize]{subcaption}

\newtheorem{theorem}{Theorem}

\newtheorem{lemma}{Lemma}
\newtheorem{corollary}{Corollary}

\newtheorem{notation}{Notation}

\newtheorem{definition}{Definition}

\newtheorem{problem}{Problem}

\DeclareMathOperator{\argmin}{arg\,min}
\title{The Conditioning of Hybrid Variational Data Assimilation}

\begin{document}
\author[1]{Shaerdan Shataer\thanks{oz925466@reading.ac.uk}}
\author[1,2]{Amos S. Lawless}
\author[1,2]{Nancy K. Nichols}
\affil[1]{School of Mathematical, Physical and Computational Sciences, University of Reading, Reading, UK}
\affil[2]{National Centre for Earth Observation, Reading, UK}
\maketitle

\abstract{In variational assimilation, the most probable state of a dynamical system {under Gaussian assumptions for the prior and likelihood} can be found by solving a least-squares minimization problem . In recent years, we have seen the popularity of hybrid variational data assimilation methods for Numerical Weather Prediction. In these methods, the prior error covariance matrix is a weighted sum of a climatological part and a flow-dependent ensemble part, the latter being rank deficient. The nonlinear least squares problem of variational data assimilation is solved using iterative numerical methods, and the condition number of the Hessian is a good proxy for the convergence behavior of such methods. In this paper, we study the conditioning of the least squares problem in a hybrid four-dimensional variational data assimilation (Hybrid 4D-Var) scheme by establishing bounds on the condition number of the Hessian. In particular, we consider the effect of the ensemble component of the prior covariance on the conditioning of the system. Numerical experiments show that the bounds obtained can be useful in predicting the behavior of the true condition number and the convergence speed of an iterative algorithm
	}
	
	
	\maketitle
	

	\section{Introduction}\label{sec1}
{In weather forecasting, we use mathematical and numerical models to describe the system dynamics of the ocean and atmosphere. These models are highly nonlinear and often sensitive to noise in initial conditions. Because of the nonlinearity and instability, random perturbations in the initial data and errors in the model amplify rapidly through time, producing unreliable predictions \cite{nichols2010mathematical,gratton2007approximate}. In variational data assimilation, the goal is to find the maximum Bayesian a posteriori estimate of the system state from which to initialize the model. Major operational centres worldwide have adopted the four-dimensional variational assimilation scheme (4D-Var) for environmental forecasting in recent years. Similar applications arise in other fields such as physics, biology, and economics \cite{haben2011conditioning,bannister2017review,nichols2010mathematical}. }
 
	In 4D-Var, we aim to obtain an optimal initial state variable (conventionally called the 'analysis') by solving a nonlinear least squares  problem, in which we try to find the best fit between a set of observations over a time window and an \emph{a priori} estimate of the state at the start of the window, known as the background. {In the case where observations are only given at one time, the method becomes three-dimensional variational assimilation, or 3D-Var.} We assume that the background state and observations have Gaussian, unbiased errors, with covariance matrices $B$ and $R$ respectively. Traditional variational data assimilation methods have used a climatological estimate\cite{bannister2017review} for the  background error covariance matrix $B$, where flow-dependent ensemble information is not incorporated into the system\cite{smith2017estimating,smith2018treating}. Recent developments utilise a hybrid approach, in which an ensemble background error covariance matrix is estimated with ensemble members and then combined with the climatological part. This has a clear advantage of bringing in the variability of the system and updates the statistics in each prediction window, giving flow-dependent information that can improve the accuracy of predictions \cite{bannister2017review}. However, due to computational restrictions, the affordable number of ensemble members is normally small, which leads to a rank deficient ensemble background error covariance matrix.  
	The method of combining the ensemble parts with the conventional 4D-Var is called Hybrid 4D-Var. In this method,  $B$ is given by $B = (1-\beta)B_{0}+\beta P_{f}$. Here, $B_{0}, \ P_{f}$ are the climatological background error covariance and the ensemble error covariance matrix, and $(1-\beta),\ \beta$ are their weights, where $\beta$ is a scalar. For the details of the hybrid method, we refer readers to the review paper of Bannister \cite{bannister2017review}.
	
{In this paper, we are especially interested in establishing the relationship between the conditioning of Hybrid 4D-Var and the weight $\beta$ on $P_{f}$.  The 4D-Var problem is usually solved using iterative gradient methods, such as conjugate gradient or Quasi-Newton methods\cite{bannister2017review,nichols2010mathematical}.   The condition number of the Hessian can be used to estimate the number of iterations required for convergence\cite{golub2013matrix, tabeart2018conditioning}.  In addition, the condition number also reveals the sensitivity of the minimisation problem with respect to random noise\cite{golub2013matrix}. 
  Here we establish that, in Hybrid 4D-Var, a transition point exists where the condition number of $B$ sharply increases with the weight on $P_{f}$. Since $P_{f}$ is rank deficient \cite{bannister2017review,smith2017estimating,smith2018treating,nichols2010mathematical},
  adding the ensemble background error covariance matrix may cause difficulties for solving the nonlinear least-squares minimisation problem \cite{bannister2017review,smith2017estimating,smith2018treating,nichols2010mathematical}
 and an adequate preconditioning scheme becomes desirable. This is typically achieved through a Control Variable Transformation (CVT) \cite{bannister2017review,nichols2010mathematical}. CVT uses a decomposition of $B_{0}$ and $P_{f}$ to transform the state variable such that the conditioning of the Hessian matrix is improved \cite{bannister2017review,nichols2010mathematical}.   
  Implementation details of CVT are frequently described in previous works on 4D-Var\cite{tabeart2022new,haben2011conditioning,smith2017estimating,smith2018treating} and Hybrid 4D-Var\cite{bannister2017review}.    
  In terms of practical applications of Hybrid 4D-Var, we note that the nonlinear least-squares problem is often linearised and solved as a sequence of linear least-squares minimisations. This is known as the incremental 4D-Var method \cite{bannister2017review,haben2011conditioning,nichols2010mathematical} and is equivalent to an approximate Gauss-Newton method\cite{lawless2005investigation}.}
	
	In practice, it is useful to understand the contribution of each component of DA, in such a way that the impact on the conditioning can be predicted when these components change. This then motivates a comprehensive theory that can predict the conditioning while separating the contribution of each error component ($B_{0},P_{f},R$), and relating it to the parameters that characterize these components. We note that both preconditioned and unpreconditioned 4D-Var are implemented by major operational centres (such as the UK Met Office). 
 Such theories have previously been constructed for conventional 4D-Var. Haben et al (2011)\cite{haben2011conditioning} established a theory to estimate the conditioning of a preconditioned 4D-Var system. This work is later developed by Tabeart et al (2018, 2021) \cite{tabeart2018conditioning,tabeart2022new} for both unpreconditioned 3D-Var and preconditioned 4D-Var with Control Variable Transformation (CVT). In Tabeart et al's studies, a bound estimation is proved for the conditioning of the system, in which the contributions of the background error covariance matrix and observation error covariance matrix are separated. The impact of each component is then associated with its characterizing parameters. The research of Haben et al\cite{haben2011conditioning} and Tabeart et al\cite{tabeart2018conditioning,tabeart2022new} are tested and analysed using small scale examples. However, they are extrapolated to justify observations in large scale real-life applications. To give a few examples, {Mercier et al\cite{mercier2019speeding} used the analysis of Haben et al\cite{haben2011conditioning} to explain the convergence behaviour of a block Krylov method for a 3D-Var application\cite{mercier2019speeding}; Desroziers et al\cite{desroziers2012accelerating} cited the same result of Haben et al \cite{haben2011conditioning} to guide their design of a preconditioned Lanczos/Conjugate Gradient algorithm. Hatfield et al\cite{hatfield2020single} cited Haben et al's analysis \cite{haben2011conditioning} to explain the effect of an increasing model error on the convergence of an incremental 4D-Var; Aabaribaoune et al \cite{aabaribaoune2020estimation} used the result of Tabeart et al\cite{tabeart2018conditioning} to analyse the convergence speed of a BFGS algorithm for solving a 3D-Var system, in the application of ozone profiling}.  

	
	In this paper, we aim to extend previous studies of Haben et al\cite{haben2011conditioning} and Tabeart et al \cite{tabeart2018conditioning,tabeart2022new} to a hybrid system. In particular, we study the impact of $P_{f}$ on the condition number of the Hessian matrix, for both unpreconditioned cases and preconditioned cases with CVT. 
	
	We outline this paper as follows. In section 2, we briefly formulate the problem of Hybrid 4D-Var and introduce the Hessian matrix; in section 3, we establish the theory of conditioning for unpreconditioned Hybrid 4D-Var and preconditioned Hybrid 4D-Var with CVT; in sections 4 to 6, we provide multiple numerical experiments to illustrate the theories and analyse the conditioning; in section 7, we show how the behaviour of the condition number of the Hessian predicted by our theory is reflected in the convergence speed of a conjugate gradient algorithm. Section 8 gives a general summary of the results.  
	
	\section{Problem Formulation}\label{sec2}
	In this section we introduce the Hybrid 4D-Var method, the incremental method, the preconditioning technique of Control Variable Transformation, and the Hessian matrix associated with the unpreconditioned and preconditioned least-squares problems. 
	
	\subsection{A General Formulation of the Hybrid 4D-Var}
	A general formulation of the 4D-Var is given by Problem \ref{Problem 1}. 
	\begin{problem}\label{Problem 1}
		Solve for the optimal initial state $\boldsymbol{x_{a}}$ by solving a minimisation problem given by,\begin{align}\label{4d-var}
			\begin{cases}
				&\boldsymbol{x_{a}} = \argmin_{\boldsymbol{x_{0}}\in \mathbb{R}^{n}} 	\mathcal{J}(\boldsymbol{x_{0}}), \\ &\mathcal{J}(\boldsymbol{x_{0}}) :=\frac{1}{2}||\boldsymbol{x_{0}}-\boldsymbol{x^{b}}||_{B^{-1}}^{2} +\frac{1}{2} \sum_{i=0}^{N}||\boldsymbol{y_{i}}-\mathcal{H}_{i}(\boldsymbol{x_{i}})||^{2}_{R_{i}^{-1}},\\
				&\boldsymbol{x_{i}}\in \mathbb{R}^{n}, \ \boldsymbol{y_{i}}\in \mathbb{R}^{p}, \ B \in \mathbb{R}^{n,n}, \ R_{i} \in \mathbb{R}^{p,p}, \ \mathcal{H}_{i} : \mathbb{R}^{n}\to \mathbb{R}^{p},
			\end{cases}
		\end{align}
		where $\mathcal{H}_{i}$ is the nonlinear observational operator; $\boldsymbol{y_{i}}$ is the vector of observational measurements, taken at time $t_{i}$; $\boldsymbol{x_{i}}$ is the state variable at time $t_{i}$, given by,
		\begin{align}\label{forward 1}
			\boldsymbol{x_{i}}= \mathcal{M}_{i,i-1}(\boldsymbol{x_{i-1}})=\mathcal{M}_{i,0}(\boldsymbol{x_{0}}),
		\end{align}
		$\mathcal{M}_{i,i-1}: \mathbb{R}^{n}\to \mathbb{R}^{n}$ is the nonlinear model, and $\mathcal{M}_{i,0}=\mathcal{M}_{i,i-1}\cdot\mathcal{M}_{i-1,i-2}\cdots\mathcal{M}_{1,0}$ is a direct product (composition) of them. The vector $\boldsymbol{x^{b}}\in \mathbb{R}^{n}$ is the prior information given by the model at $t_{0}$, known as the background state.  { A simplified problem, given by 3D-Var, is a special case of the general 4D-Var problem where $N=0$, meaning there are only observations at time $t_{0}$.}
	\end{problem}

 {In Hybrid 4D-Var and 3D-Var, we follow the formulation given by (\ref{4d-var}), but replace the background error covariance matrix $B$ with $B = (1-\beta)B_{0}+\beta P_{f}$.}
We note that the problem 
 involves solving a nonlinear least squares problem, and {it is} difficult to implement directly when the problem is large scale. Instead, this is often replaced with a linearised incremental formulation, which we describe next.

	\subsection{Incremental Hybrid 4D-Var, Control Variable Transformation and The Hessian Matrix}
	{In practice, especially in NWP, Problem \ref{Problem 1} is often solved using the incremental method. 
 In incremental 4D-Var, the nonlinear least squares problem is replaced with a sequence of linear least squares problems with a cost function of}
	\begin{align}\label{incremental cost function}
		\mathcal{J}(\delta \boldsymbol{x}_{0}^{k}) = \frac{1}{2}||\delta \boldsymbol{x}_{0}^{k} - \delta \boldsymbol{x}_{b}^{k}||_{B^{-1}}^{2} + \frac{1}{2}\sum_{i=0}^{N}||\boldsymbol{d}_{i}^{k} - H_{i}\delta \boldsymbol{x}_{i}^{k}||_{R_{i}^{-1}},
	\end{align}
	where the vector $\boldsymbol{d}_{i}^{k} $ is known as the innovation, defined by $\boldsymbol{d}_{i}^{k} = y_{i}-\mathcal{H}_{i}(\boldsymbol{x}_{i}^{k})$; the linear operator $H_{i}$ is the Jacobian of $\mathcal{H}_{i}$, and the vector $\delta \boldsymbol{x}_{b}^{k}$ is the increment, given by $\delta \boldsymbol{x}_{b}^{k} = \boldsymbol{x}_{0}^{k} - \boldsymbol{x}_{b}^{k}$; the vector $\delta \boldsymbol{x}_{i}^{k}$ is computed from $\delta \boldsymbol{x}_{i}^{k} = M_{i,0}\delta \boldsymbol{x}_{0}^{k}$. The model operator $M_{i,0}$ is given by $M_{i,0} = \prod_{j=1}^{i}M_{j}$, where $M_{j}$ is the Jacobian of $\mathcal{M}_{j,j-1}$. At each outer iteration $k$ (the outer loop), we minimise (\ref{incremental cost function}) to solve for the increment, and then $\boldsymbol{x}_{0}^{k}$ is updated for the next iteration. 
	 
	We now introduce the Hessian matrix of the cost function $\mathcal{J}$ in Problem \ref{Problem 1}, which is given by\cite{tabeart2022new,haben2011conditioning},
	\begin{align*}
		{S_{4D}} = ((1-\beta)B_{0}+\beta P_{f})^{-1} + \sum_{i=0}^{N}({H}_{i}M_{i,0})^{T}R_{i}^{-1}{H}_{i}M_{i,0}.
	\end{align*}  
	
	The sum of the matrix product above can be written in a simple compact form. 
	\begin{notation}\label{notation 1}
		The matrix 
  $\hat{{H}}$ is the general observation operator, given by \begin{align*}
      \hat{{H}}:= \left[{H}_{0}^{T}, ({H}_{1}{M}_{1,0})^{T},\cdots,({H}_{N}{M}_{N,0})^{T}\right]^{T}\in \mathbb{R}^{p(N+1),n}
  \end{align*}, where $\hat{R}$ is a block diagonal matrix with its $i$th diagonal block  given by $R_{i}$.
	\end{notation}
	Following Notation \ref{notation 1}, we can rewrite the Hessian matrix as follows\cite{tabeart2022new,haben2011conditioning},
	\begin{align}\label{Hessian formula}
		{S_{4D}} = ((1-\beta)B_{0}+\beta P_{f})^{-1} + \hat{{H}^{T}}\hat{R}^{-1}\hat{{H}}.
	\end{align}  
In the case of Hybrid 3D-Var, we recall that it is a special case of 4D-Var with $N=0$ and its Hessian is given by
    \begin{align}\label{Hessian 3dvar}
        {S_{3D}} = ((1-\beta)B_{0}+\beta P_{f})^{-1} + {{{H}_{0}^{T}}{R_{0}}^{-1}{{H}_{0}}}.
    \end{align}
	
	In applications, preconditioning techniques are often applied to improve the conditioning of the system.
	Control Variable Transformation (CVT) is one of the popular preconditioning techniques. The detail of CVT is described by Nichols (2010) \cite{nichols2010mathematical}, Bannister (2017) \cite{bannister2017review} and Buehner (2005) \cite{buehner2005ensemble}. In this approach, we utilise factorizations of $B_{0}$ and $P_{f}$, given by
	\begin{align}
		B_{0} = U^{T}U, \ P_{f} = X_{f}^{T}X_{f}, \ \text{where} \ U= B_{0}^{1{/}2} \ \text{and} \\ \ X_{f} = \frac{1}{\sqrt{m-1}}\left[\boldsymbol{x}_{1}-\boldsymbol{\bar{x}},\boldsymbol{x}_{2}-\boldsymbol{\bar{x}},\cdots, \boldsymbol{x}_{m}-\boldsymbol{\bar{x}}\right], \nonumber
	\end{align}
 and $\boldsymbol{x}_{1},\boldsymbol{x}_{2},\cdots, \boldsymbol{x}_{m}$ are ensemble members, $m$ is the number of samples and $\bar{\boldsymbol{x}}$ is the ensemble mean. These matrices are then used to transform the state variable as follows \cite{bannister2017review},
	\begin{align}
		U_{h} \delta \boldsymbol{v} =  \delta \boldsymbol{x}, \ \text{where} \ U_{h} = \left[\sqrt{(1-\beta)}U \ \sqrt{\beta}X_{f}\right] \in \mathbb{R}^{n,n+m}.
	\end{align}
	Here $\delta \boldsymbol{x} \in \mathbb{R}^{n}$ is the increment of the state variable and $\delta \boldsymbol{v}$ is the increment of the control variable. Applying this transform to (\ref{incremental cost function}) leads to a new cost function that reads
	\begin{align}
		J(\delta \boldsymbol{v}_{0}^{k}) =  \frac{1}{2}||\delta \boldsymbol{v}_{0}^{k} - \delta \boldsymbol{v}_{b}^{k}||_{I_{n+m}} + \frac{1}{2}\sum_{i=0}^{N}||y_{i}-\mathcal{H}_{i}(\boldsymbol{x}^{k}) - H_{i}M_{i,0}U_{h}\delta \boldsymbol{v}_{i}^{k}||_{R_{i}^{-1}}. 
	\end{align} 
	A direct calculation then yields the Hessian of $J(\delta \boldsymbol{v}_{0}^{k})$ with respect to $\delta \boldsymbol{v}_{0}^{k}$ as
	\begin{align}
		{S_{P4D}} = I_{n+m} + U_{h}^{T}\hat{H}^{T}\hat{R}^{-1}\hat{H}U_{h}.
	\end{align}
	In the following sections we will demonstrate that CVT prevents the condition number from going to infinity when $\beta$ (the weight of the ensemble part) approaches 1. In addition, we also note that the adjoint of $U_{h}$ does not need to be computed explicitly. This is discussed in detail by Smith et al\cite{smith2017estimating} and Bannister\cite{bannister2017review}.
	
	\section{Theory of the conditioning of Hybrid 4D-Var}
	
	To better analyse the convergence of the nonlinear least squares problem of Hybrid 4D-Var, we aim to establish a set of theories that can predict changes in the conditioning of the system prompted by varying parameters (such as correlation length scale, error variance, etc). Typically, we are also keen to understand the impact of the rank deficient ensemble part on the conditioning. These motivate us to develop an estimation of the condition number that is informative of the actual conditioning of the system. In order to achieve such a goal, we use spectral theories to establish bounds for the condition number of the Hessian matrix. 
	
	We outline the structure of this section as follows. We start by introducing some pre-established results, then extend them to the Hybrid method. We will discuss the unpreconditioned cases and the preconditioned cases separately.  
	\vspace{-0.5 cm}
	\subsection{
 Eigenvalues and Conditioning}
	We begin with a brief review of previous work by Tabeart et al \cite{tabeart2018conditioning} and some fundamental eigenvalue inequalities\cite{wilkinson1971algebraic}. We will extend these results to Hybrid 4D-Var. In the scope of this paper we always assume that $B$ and $R$ are symmetric.
	
	\begin{notation}
		Let $\lambda_{k}(A)$ be the $k$th largest eigenvalue of a matrix $A\in \mathbb{R}^{n,n}$,  $\lambda_{1}(A), \lambda_{n}(A)$ be the largest and smallest eigenvalues of $A$, and $\kappa(A)$ be the condition number of $A$.
	\end{notation}
\begin{definition}
	For a symmetric positive definite matrix $A \in \mathbb{R}^{n,n}$, its condition number is defined by $\kappa({A}) =\lambda_{1}(A){/}\lambda_{n}(A) $.
	\end{definition}
	For the eigenvalues of the sum of two Hermitian matrices, H.Weyl (1912)\cite{weyl1912asymptotische} proved the following theorem: 
	\begin{theorem}\label{thm ineq sum}
		Let $A_{1},A_{2}$ be two symmetric matrices. Then the eigenvalues of $A = A_{1}+A_{2}$ satisfy the following:
		\begin{align}\label{ineq sum}
			\lambda_{k}(A_{1})+\lambda_{n}(A_{2})\leq\lambda_{k}(A)\leq \lambda_{k}(A_{1})+\lambda_{1}(A_{2}).
		\end{align} 
	\end{theorem} 
The inequality for products is given by Wang and Zhang as follows\cite{wang1992some},
	\begin{theorem}\label{thm ineq product}
		Let $A_{1},A_{2} \in R^{n,n}$ be positive semidefinite Hermitian matrices. Then
		\begin{align}\label{ineq product}
			\max\left[\lambda_{1}(A_{1})\lambda_{n}(A_{2}),\lambda_{1}(A_{n})\lambda_{1}(A_{2})\right]\leq	\lambda_{1}(A_{1}A_{2}) \leq \lambda_{1}(A_{1})\lambda_{1}(A_{2}).
		\end{align} 	
	\end{theorem}
	
Applying these results to the Hessian matrix of 3D-Var, Tabeart et al\cite{tabeart2018conditioning} established that,  
	\begin{theorem}\label{thm jemima 1}
		Let ${S_{3D}} = B^{-1}+{H_{0}^{T}R_{0}^{-1}H_{0}}$ and {given} that $B,R$ are symmetric positive definite, the condition number of ${S_{3D}} $ is bounded as follows,
		\begin{align}\label{ineq 2}
  \begin{cases}
&\kappa({S_{3D}})\ge\max\left[\frac{\kappa({B})}{1+\lambda_{1}(B)\lambda_{1}(H_{0}^{T}R_{0}^{-1}H_{0})},\frac{1+\lambda_{1}(B)\lambda_{1}(H_{0}^{T}R_{0}^{-1}H_{0})}{\kappa({B})}\right]\\ &\kappa({S_{3D}}) \leq (1+\lambda_{n}(B)\lambda_{1}(H_{0}^{T}R_{0}^{-1}H_{0}))\kappa({B}).
\end{cases}
		\end{align}
	\end{theorem}

	\subsection{Conditioning of Unpreconditioned Hybrid 4D-Var}
	In Hybrid 4D-Var, the background error covariance matrix $B$ is replaced by a weighted sum of the static and ensemble background error covariance matrix, such that $B = (1-\beta)B_{0} + \beta P_{f}$. Although we can still use (\ref{ineq 2}) to estimate the condition number of the Hessian, we note that it does not have the capacity to separate the contribution of $P_{f}$ and predict the condition number as the weight of $P_{f}$ grows. This motivates us to establish theories that are specific to 
the hybrid case of 4D-Var.

	\begin{lemma}\label{lemma 1}
		
		Assume that $P_{f}$ is rank deficient and let $B = (1-\beta) B_{0} +\beta P_{f}$. Then the extreme eigenvalues of $B$ satisfy the following:
		\begin{align}\label{bound lambdaB}
			\begin{cases}
				\max \left[(1-\beta)\lambda_{1}(B_{0}), \ \beta\lambda_{1}(P_{f})+(1-\beta)\lambda_{n}(B_{0})\right]\leq\lambda_{1}(B)\leq (1-\beta)\lambda_{1}(B_{0})+\beta\lambda_{1}(P_{f}), \\
				(1-\beta)\lambda_{n}(B_{0})\leq\lambda_{n}(B)\leq \min\left[(1-\beta)\lambda_{1}(B_{0}), \ \beta\lambda_{1}(P_{f})+(1-\beta)\lambda_{n}(B_{0})\right].
			\end{cases}
		\end{align} 
	\end{lemma}
	\begin{proof}
		The conclusion follows from Theorem \ref{thm ineq sum} and that $\lambda_{n}(P_{f}) = 0$ (since $P_{f}$ is rank deficient).
	\end{proof}	
	
	\begin{lemma}\label{lemma kappaB}
		{Given} $B_{0},P_{f}$ are two symmetric matrices, let $B = (1-\beta)B_{0}+\beta P_{f}$, and assuming that $P_{f}$ is rank deficient, the condition number of $B$ is then bounded as follows,
		\begin{align}\label{bound kappaB}
			\max\left[\frac{1}{\kappa({B_{0}})}+\frac{\beta\lambda_{1}(P_{f})}{(1-\beta)\lambda_{1}(B_{0})},\left(\frac{1}{\kappa({B_{0}})}+\frac{\beta\lambda_{1}(P_{f})}{(1-\beta)\lambda_{1}(B_{0})}\right)^{-1}\right]	\leq\kappa({B})\leq \kappa({B_{0}})\left(1 + \frac{\beta\lambda_{1}(P_{f})}{(1-\beta)\lambda_{1}(B_{0})}\right).
		\end{align} 
	\end{lemma}
	\begin{proof}
		By the definition of $\kappa({B})$ and Lemma \ref{lemma 1}, we find that,
		\begin{align*}
			\kappa({B}) = \frac{\lambda_{1}(B)}{\lambda_{n}(B)} \leq \frac{(1-\beta)\lambda_{1}(B_{0})+\beta\lambda_{1}(P_{f})}{(1-\beta)\lambda_{n}(B_{0})} = \kappa({B_{0}}) + \frac{\beta\lambda_{1}(P_{f})}{(1-\beta)\lambda_{n}(B_{0})},	
		\end{align*}
		and similarly,
		\begin{align*}
			\kappa({B}) \geq  \max\left[\frac{(1-\beta)\lambda_{n}(B_{0})+\beta\lambda_{1}(P_{f})}{(1-\beta)\lambda_{1}(B_{0})}, \frac{(1-\beta)\lambda_{1}(B_{0})}{(1-\beta)\lambda_{n}(B_{0})+\beta\lambda_{1}(P_{f})}\right] = \\ \max\left[\frac{1}{\kappa({B_{0}})}+\frac{\beta\lambda_{1}(P_{f})}{(1-\beta)\lambda_{1}(B_{0})},\left(\frac{1}{\kappa({B_{0}})}+\frac{\beta\lambda_{1}(P_{f})}{(1-\beta)\lambda_{1}(B_{0})}\right)^{-1}\right].
		\end{align*}
	\end{proof}
	We highlight that the condition number of $B$ diverges to infinity as $\beta \to 1$. This is reflected by the lemma above as the lower bound of $\kappa(B)$ goes to infinity.
	
 {We now use these results to obtain a bound on the condition number in Theorem \ref{thm unpreconditioned main}. We emphasize that the conclusion of these results cannot be directly derived from Theorem \ref{thm jemima 1} by simply replacing $B$ with $(1-\beta)B_{0} + \beta P_{f}$ in the upper and lower bounds. In fact, such an effort leads to a less sharp bound that does not provide analytical or theoretical insight on how the bound changes with $\beta$. Moreover, this approach does not reveal how the bound responds to changes in physical parameters related to $B_{0}$ and $P_{f}$. 
Using different strategies leads to a bound that clearly separates the contribution of each component ($B_{0},P_{f}$) and their weights, making it possible to analyze their impact on the condition number and their interaction as well.}

	\begin{notation}
		For simplicity of presentation, we introduce $\gamma_{z}, \ \Gamma_{z}$ to represent the lower and upper bound of any $z \in \mathbb{R}$.
	\end{notation}	
	\begin{theorem}\label{thm unpreconditioned main}
		Recall that $S_{4D} = B^{-1}+\hat{H}^{T}\hat{R}^{-1}\hat{H}$, $B = (1-\beta)B_{0}+\beta P_{f}$, and let,
		\begin{align*}
			&\Gamma_{\lambda_{n}(B)} = \min ( (1-\beta)\lambda_{n}(B_{0})+\beta\lambda_{1}(P_{f}),  (1-\beta)\lambda_{1}(B_{0}) ), \ \\  &\gamma_{\kappa({B})} = \max\left[\frac{1}{\kappa({B_{0}})}+\frac{\beta\lambda_{1}(P_{f})}{(1-\beta)\lambda_{1}(B_{0})},\left(\frac{1}{\kappa({B_{0}})}+\frac{\beta\lambda_{1}(P_{f})}{(1-\beta)\lambda_{1}(B_{0})}\right)^{-1}\right], \\ &\Gamma_{\kappa({B})}  = \kappa({B_{0}})\left(1 + \frac{\beta\lambda_{1}(P_{f})}{(1-\beta)\lambda_{1}(B_{0})}\right).
		\end{align*}  Given that $B,R$ are symmetric, the condition number of $S_{4D}$ is then bounded as follows,
		\begin{align}\label{ineq unpreconditioned main}
  \begin{cases}
   		\kappa({S_{4D}}) \geq \max	\left[	\frac{1	}{\Gamma_{\kappa({B})}} + (1-\beta)\lambda_{n}(B_{0})\lambda_{1}(\hat{H}^{T}\hat{R}^{-1}\hat{H}),\ \left(	\frac{1	}{\gamma_{\kappa({B})}} + \Gamma_{\lambda_{n}(B)}\lambda_{1}(\hat{H}^{T}\hat{R}^{-1}\hat{H})\right)^{-1}, \ 1 \right]   \\   		\kappa({S_{4D}}) \leq  \Gamma_{\kappa({B})}+\left((1-\beta)\lambda_{1}(B_{0})+\beta\lambda_{1}(P_{f})\right) \lambda_{1}(\hat{H}^{T}\hat{R}^{-1}\hat{H}).   
  \end{cases}
		\end{align}
	\end{theorem}
	\begin{proof}
		In Theorem \ref{thm jemima 1}, replacing $H_{0}^{T}R^{-1}H_{0}$ with $\hat{H}\hat{R}^{-1}\hat{H}$, then the left hand side of Inequality (12) can be written 
		\begin{align*}
			\max\left[\frac{1}{\kappa({B})}+\lambda_{n}(B) \lambda_{1}(\hat{H}^{T}\hat{R}^{-1}\hat{H}),\ \left(\frac{1}{\kappa({B})}+\lambda_{n}(B) \lambda_{1}(\hat{H}^{T}\hat{R}^{-1}\hat{H})\right)^{-1} \right]\leq 		\kappa({S_{4D}}).
		\end{align*}
		
  For the upper bound, we note that the upper bound 
  can be written as follows,
		\begin{align*}
			 		\kappa({S_{4D}})\leq \kappa({B}) + \lambda_{1}(B)\lambda_{1}(\hat{H}^{T}\hat{R}^{-1}\hat{H}).
		\end{align*}
		Substituting $\kappa({B})$ and $\lambda_{1}(B)$ with their upper bounds, this then produces the upper bound on $ 		\kappa({S_{4D}})$.
	\end{proof}
 We note that the same bound can be derived by using eigenvalue 
 inequalities.	
 Crucially, the upper bound in Theorem \ref{thm unpreconditioned main} does not require any explicit computation of $S_{4D}$. We only need to compute the largest eigenvalues of $\lambda_{1}(B_{0}),\lambda_{1}(P_{f})$ and the condition number of $B_{0}$.

 In the case of a special observation operator we can simplify the Hessian in the 3D-Var case as follows.	 
	\begin{corollary} \label{coro 2}
  Assuming that each row of $H_{0}$ has one unit entry, with all other entries being zero, and that $R_{0}$ is diagonal such that $R_{0} = \sigma_{R_{0}}^{2}I_{p}$, where $I_{p} \in \mathbb{R}^{p,p}$ is the identity matrix, then the bounds in Theorem \ref{thm unpreconditioned main} in the case of 3D-Var simplify to:
		\begin{align}\label{ineq coro 2}
			\max \left[	\frac{1	}{\Gamma_{\kappa({B})}} + \frac{(1-\beta)\lambda_{n}(B_{0})}{\sigma_{R_{0}}^{2}},\ \left(	\frac{1	}{\gamma_{\kappa({B})}} + \frac{\Gamma_{\lambda_{n}(B)}}{\sigma_{R_{0}}^{2}}\right)^{-1}\right] \leq	 		\kappa({S_{3D}})  \leq  \Gamma_{\kappa({B})}+\frac{(1-\beta)\lambda_{1}(B_{0})+\beta\lambda_{1}(P_{f})}{\sigma_{R_{0}}^{2}}.
		\end{align}
	\end{corollary}
	\begin{proof}
		Given that $R_{0} = \sigma_{R_{0}}^{2}I_{p}$, and with the assumption on $H$, the largest eigenvalue of ${H}_{0}^{T}{R}_{0}^{-1}{H}_{0}$ is always $\lambda_{1}({H}_{0}^{T}{R}_{0}^{-1}{H}_{0}) = \sigma_{R_{0}}^{-2}$. Replacing relevant terms in Bounds (\ref{ineq unpreconditioned main}), we then reach the conclusion.
	\end{proof}
	
We emphasise that the upper bound on $\kappa(S_{4D})$ diverges to infinity as $\beta \to 1$. In such a case, the background error covariance matrix is dominated by $P_{f}$. However, the theory does not predict the same 
 behaviour for the lower bound. 
 In terms of 
 of the observation part, we observe that the number $p$ of observation points does not alter the bounds; thus the theory can not predict the behaviour of $\kappa({S_{4D}})$ when $p$ varies.  
	
	Although unpreconditioned 4D-Var is still in use for real-life applications\cite{tabeart2018conditioning}, it is now a common practice to implement CVT for Hybrid 4D-Var. In the next section, we focus on developing similar theories for such cases.

	\subsection{Conditioning of Preconditioned Hybrid 4D-Var Method with CVT}\label{CVT theory}
	The preconditioning of Hybrid 4D-Var is broadly adopted in major operational centres, and the impact of different error components on the conditioning is important to determine the convergence of numerical schemes for incremental Hybrid 4D-Var, or more generally, for solving the nonlinear least squares problem in Hybrid 4D-Var. In this section we prove two versions of the bounds for $ 		\kappa({S_{4D}})$ with different strategies. Firstly we remind readers some useful notation as follows: let $K := \hat{H}^{T}\hat{R}\hat{H}$, $U := B_{0}^{{1}{/}{2}}$ and $X_{f} := \frac{1}{\sqrt{m-1}}\left[\boldsymbol{x}_{1}-\bar{\boldsymbol{x}},\boldsymbol{x}_{2}-\bar{\boldsymbol{x}},\cdots, \boldsymbol{x}_{m}-\bar{\boldsymbol{x}}\right]$,  where $\boldsymbol{x}_{1},\boldsymbol{x}_{2},\cdots,\boldsymbol{x}_{m}$ is the set of ensemble members and $\bar{\boldsymbol{x}}$ is the ensemble mean. 

 {We recall that the motivation behind our approach to deriving the bound is to separate the contribution of each component in the Hessian, particularly the components of $B_{0}$ and $P_{f}$ and their weights. We are especially keen to find out how the ratio of $(1-\beta)||B_{0}||_{2}{/}(\beta||P_{f}||_{2})$ influences the bound, as this ratio is linked to the balancing of the climatological part and the ensemble part. Driven by this motivation, we discovered a decomposition of the matrix product $U_{h}^{T}\hat{H}^{T}\hat{R}^{-1}\hat{H}U_{h}$ that yields a useful max function which is directly related to this ratio. The detail is given by the proof of Theorem \ref{thm preconditioned main}.}
	\begin{theorem}\label{thm preconditioned main}
  The condition number of the Hessian matrix of the Hybrid 4D-Var method satisfies
		\begin{align}\label{ineq preconditioned main}
			\begin{cases}
			{\kappa(S_{P4D})} \leq 1 +\sqrt{(\beta-\beta^{2})\lambda_{1}(B_{0})\lambda_{1}(P_{f})\lambda_{1}(K^{2})} +  \max \left[(1-\beta)\lambda_{1}(B_{0})\lambda_{1}(K), \  \beta\lambda_{1}(P_{f})\lambda_{1}(K)\right],\\
			{\kappa(S_{P4D})} \geq 1 + \max\left[  (1-\beta)\lambda_{1}(K)\lambda_{n}(B_{0}), \  \sqrt{(\beta - \beta^{2})\lambda_{n}(B_{0})\lambda_{1}(P_{f})\lambda_{1}(K^{2})}\right]
				\end{cases}
		\end{align}
	\end{theorem}
	\begin{proof}
		We know that the Hessian matrix of the Hybrid 4D-Var can be expressed as follows,
		\begin{align}
			\begin{cases}
				&{S_{P4D}} = I_{n+m} + U_{h}^{T}\hat{H}^{T}\hat{R}^{-1}\hat{H}U_{h},\\
				&U_{h} = \left[\sqrt{1-\beta}U \ \  \sqrt{\beta}X_{f}\right].
			\end{cases}
		\end{align}
    Given that $U_{h}^{T}\hat{H}^{T}\hat{R}^{-1}\hat{H}U_{h}$ is rank deficient, it is immediate that 
    \begin{align}\label{condition number equality}
			{\kappa(S_{P4D})} = 1 + \lambda_{1}(U_{h}^{T}\hat{H}^{T}\hat{R}^{-1}\hat{H}U_{h}).
    \end{align}
    Substituting $U_{h}$ in $S_{P4D}$, we derive
		\begin{align}\small
			{S_{P4D}} &= I_{n+m} + \begin{pmatrix}
				&\sqrt{1-\beta}U^{T} \\ & \sqrt{\beta}X_{f}^{T}
			\end{pmatrix} K \begin{pmatrix}
				&\sqrt{1-\beta}U& \sqrt{\beta}X_{f}\end{pmatrix} 
			\\
			&= I_{n+m} + \begin{pmatrix}
				&(1-\beta)U^{T}KU & 0 \\
				&0& \beta X_{f}^{T}KX_{f}
			\end{pmatrix} + \begin{pmatrix}
				&0& \sqrt{\beta - \beta^{2}}U^{T}KX_{f} \\
				&\sqrt{\beta - \beta^{2}}X_{f}^{T}KU & 0
			\end{pmatrix}\\
                &:= I_{n+m} +  A_1 + A_2.
		\end{align}
		Since $U_{h}^{T}\hat{H}^{T}\hat{R}^{-1}\hat{H}U_{h} = A_{1}+A_{2}$, using Theorem \ref{thm ineq sum}, we get
		\begin{align}
			\max \left[\lambda_{1}(A_{1})+\lambda_{n+m}(A_{2}), \  \lambda_{n+m}(A_{1})+\lambda_{1}(A_{2}) \right]\leq\lambda_{1}(U_{h}^{T}\hat{H}^{T}\hat{R}^{-1}\hat{H}U_{h})\leq\lambda_{1}(A_{1})+\lambda_{1}(A_{2}).
		\end{align}
    We note that
\begin{align}
	\lambda_{1}(U^{T}KU) = \lambda_{1}(B_{0}K),   \lambda_{1}(X_{f}^{T}KX_{f}) =\lambda_{1}(P_{f}K), 
	\lambda_{1}(U^{T}KX_{f}X_{f}^{T}KU) = \lambda_{1}(B_{0}P_{f}K^{2}).	
	\end{align}
  Then, applying Theorem \ref{thm ineq product} to the matrix products, 
  we derive that
		\begin{align}
			&\lambda_{1}(A_{1}) = \max \left[(1-\beta)\lambda_{1}(U^{T}KU), \beta\lambda_{1}(X_{f}^{T}KX_{f})\right] \leq \max \left[(1-\beta)\lambda_{1}(B_{0})\lambda_{1}(K), \ \beta\lambda_{1}(P_{f})\lambda_{1}(K)\right],\\ & \lambda_{1}(A_{2}) = \sqrt{\lambda_{1}\left[(\beta-\beta^{2})U^{T}KX_{f}X_{f}^{T}KU\right]} \leq \sqrt{(\beta-\beta^{2})\lambda_{1}(B_{0})\lambda_{1}(P_{f})\lambda_{1}(K^{2})}.
		\end{align}
		So
		\begin{align}
			{\kappa(S_{P4D})} \leq 1 + \max \left[(1-\beta)\lambda_{1}(B_{0})\lambda_{1}(K), \ \beta\lambda_{1}(P_{f})\lambda_{1}(K)\right] + \sqrt{(\beta-\beta^{2})\lambda_{1}(B_{0})\lambda_{1}(P_{f})\lambda_{1}(K^{2})}.
		\end{align}
	In the lower bound, we use that 
	\begin{align}\label{lower bound of the product}
	\max \left[\lambda_{1}(A_{1}), \lambda_{1}(A_{2}) \right]	\leq	\max \left[\lambda_{1}(A_{1})+\lambda_{n+m}(A_{2}), \  \lambda_{n+m}(A_{1})+\lambda_{1}(A_{2}) \right]\leq\lambda_{1}(U_{h}^{T}\hat{H}^{T}\hat{R}^{-1}\hat{H}U_{h}).
		\end{align}
Applying Theorem 2 
then yields (we note that $K$ is rank deficient, such that $\lambda_{n}(K) = 0$) 
\begin{align}
	&\lambda_{1}(U^{T}KU)\geq \max\left[\lambda_{n}(B_{0})\lambda_{1}(K),\lambda_{1}(B_{0})\lambda_{n}(K)\right] =\lambda_{n}(B_{0})\lambda_{1}(K),  \\ &\lambda_{1}(X_{f}^{T}KX_{f}) \geq \max\left[ \lambda_{1}(P_{f})\lambda_{n}(K),\ \lambda_{n}(P_{f})\lambda_{1}(K)\right] = 0, \\& \lambda_{1}(U^{T}KX_{f}X_{f}^{T}KU) \geq \max\left[\lambda_{1}(B_{0})\lambda_{1}(P_{f})\lambda_{n}(K^{2}),\lambda_{1}(B_{0})\lambda_{n}(P_{f})\lambda_{1}(K^{2}),\lambda_{n}(B_{0})\lambda_{1}(P_{f})\lambda_{1}(K^{2})\right] =\nonumber \\ &\lambda_{n}(B_{0})\lambda_{1}(P_{f})\lambda_{1}(K^{2}) 
	\end{align}
and giving us
		\begin{align}\label{lower bound lam A1 A2}
	\lambda_{1}(A_{1}) \geq (1-\beta)\lambda_{1}(K)\lambda_{n}(B_{0}), \ \lambda_{1}(A_{2}) \geq \sqrt{(\beta - \beta^{2})\lambda_{n}(B_{0})\lambda_{1}(P_{f})\lambda_{1}(K^{2})}.
\end{align}
By substituting $\max\left[\lambda_{1}(A_{1}),\lambda_{1}(A_{2})\right]$ into (\ref{lower bound of the product})
we conclude that  
\begin{align}
	{\kappa(S_{P4D})}\geq 1 + \max\left[  (1-\beta)\lambda_{1}(K)\lambda_{n}(B_{0}), \sqrt{(\beta - \beta^{2})\lambda_{n}(B_{0})\lambda_{1}(P_{f})\lambda_{1}(K^{2})}\right].
\end{align}	
	\end{proof}
Another version of the bounds can be derived by directly applying Theorem \ref{thm ineq product} to the product of $U_{h}KU_{h}$ instead of using the decomposition such as that in Theorem \ref{thm preconditioned main}.
	\begin{theorem}\label{thm preconditioned not main}
		Let $S_{P4D}$ be the Hessian matrix of Hybrid 4D-Var with CVT, then the condition number of $S_{P4D}$ satisfies
		\begin{align}
			1\leq {\kappa(S_{P4D})} \leq 1+\left[(1-\beta)\lambda_{1}(B_{0})+\beta\lambda_{1}(P_{f})\right]\lambda_{1}(K)
		\end{align}
	\end{theorem}
	\begin{proof}
		We note that 
		\begin{align}
			\lambda(U_{h}KU_{h}^{T}) = \lambda(U_{h}^{T}U_{h}K) = \lambda(B^{T}K).
		\end{align}
		Applying the eigenvalue inequality of the product, and using the symmetry of $B$, we find that
		\begin{align}\label{eigen inequality of products}
		\lambda_{1}(B^{T}K) \leq \lambda_{1}(B)\lambda_{1}(K).
		\end{align}
		Furthermore, the eigenvalue inequality of the sum yields
		\begin{align}\label{eigen inequality of sums}
			\lambda_{1}(B) \leq (1-\beta)\lambda_{1}(B_{0}) + \beta\lambda_{1}(P_{f}).
		\end{align}
		Applying (\ref{eigen inequality of products}) and (\ref{eigen inequality of sums}) directly to (\ref{condition number equality}), we obtain that 
		\begin{align}
	{\kappa(S_{P4D})} \leq 1+\left[(1-\beta)\lambda_{1}(B_{0})+\beta\lambda_{1}(P_{f})\right]\lambda_{1}(K)
		\end{align}
	For the lower bound, we note that applying Theorem 2 to $\lambda_{1}(B^{T}K)$ only yields
	\begin{align*}
		\lambda_{1}(B^{T}K)\geq \max \left[\lambda_{1}(B^{T})\lambda_{n}(K), \lambda_{n}(B^{T})\lambda_{1}(K) \right] 
  \ge 0.
		\end{align*}
	\end{proof}

	In terms of the effectiveness of the preconditioning, we note that the upper bounds given by Theorem \ref{thm preconditioned main} and Theorem \ref{thm preconditioned not main} are not controlled by the condition number of $B$, but by the largest eigenvalues of $B_{0}, P_{f}, K^{2}$ only. This is different from the upper bound of the unpreconditioned case. In addition, for the unpreconditioned case, $B = P_{f}$ at $\beta = 1$, indicating that $B$ becomes closer to a singular matrix as $\beta$ approaches 1. In such a case, the theory predicts that the condition number of $S_{4D}$ diverges to infinity as $\beta \to 1$. However, by implementing CVT, this divergence is eliminated.
	
	On the other hand, we note that the upper bound given by Theorem \ref{thm preconditioned main} is distinctively different from all the other versions, including the ones derived for the unpreconditioned cases. Namely, the upper bound of Theorem \ref{thm preconditioned main} contains a max function that selects the larger term of $(1-\beta)\lambda_{1}(B_{0})$ and $\beta\lambda_{1}(P_{f})$. We can thus anticipate a switching point of the upper bound as a function of $\beta$, which does not exist in other versions of the upper bound. In fact, this switching point provides useful information about the behaviour of ${\kappa(S_{P4D})}$ with respect to varying $\beta$. We will illustrate this with numerical examples in section 4. Similar to the unpreconditioned scenario, we find that these bounds do not reflect any contribution of the number $p$ of observation points, which means that they can not predict the trend of $\kappa({S_{4D}})$ with respect to a changing $p$.

	\section{experimental design}\label{section numerical experiments unpreconditioned}
Recalling the motivation of the theoretical studies in section 3, we note that the purpose of the bounds is to provide useful information on the actual condition number of the Hessian without having to compute the Hessian explicitly. However, for the bounds to be informative, we require two properties of them: the estimation given by these bounds should reflect the condition number; the bounds change with $\beta$ similarly to the condition number. The first property ensures that we can use the bound directly to estimate the condition number in practice; the second property guarantees that the bound is useful in analyzing the impact of the ensemble part on the condition number. To demonstrate the theories we use the special case of 3DVar.

	\subsection{Computing Error Covariance Matrices and the Observational Matrix}
	In this section, we give details of computing the error covariance matrices of $B_{0},P_{f}, R$ and the observational matrix $H$ on a one-dimensional grid.
	
	\begin{notation}
		Let $L_{0}, L_{ens}$ denote the correlation length scale associated with $B_{0}, P_{f}$, and $\sigma_{R_{0}}^{2}, \ \sigma_{B_{0}}^{2},\ \sigma_{P_{f}}^{2}$ be the variances of $R,B_{0},P_{f}$. Let $r$ denote the radius of a two dimensional disk $\mathcal{B}^{2}(0,r)$,  $\theta$ be the angular location of a point on the boundary of $\mathcal{B}^{2}(0,r)$ and $\theta_{i,j}$ be the angular difference between two grid point on the boundary of $\mathcal{B}^{2}(0,r)$.
	\end{notation}
	To simulate real-life applications, we use the Second-order Auto-regressive Correlation (SOAR) function to generate $B_{0}$, with a correlation length scale of $L_{0}$. The SOAR  function is used by the UK Met Office for Numerical Weather Prediction (NWP)\cite{tabeart2018conditioning,simonin2014doppler}. It is also frequently used in other DA applications\cite{gong2020inverse,cheng2021observation}. Following the formulation discussed in Tabeart et el (2018)\cite{tabeart2018conditioning}, we can compute the SOAR matrix from the following formula\cite{tabeart2018conditioning,waller2016theoretical},
	\begin{align}\label{soar}
		D_{L}({i,j}) = \left[1+2r{L^{-1}}\sin\left(\frac{\theta_{i,j}}{2}\right)\right]\exp\left[1+2r{L^{-1}}\sin\left(\frac{\theta_{i,j}}{2}\right)\right],
	\end{align}
	where $L$ is a correlation length scale.
	
	Applying (\ref{soar}), we can obtain $B_{0}$ directly by replacing $L$ with $L_{0}$, and $\theta_{i,j}$ with $2\pi{/}n$. Here we assume that the grid descritizing the boundary of $\mathcal{B}^{2}(0,r)$ is uniform and we set $r=1$. The resulting background error covariance matrix is then given by,
	\begin{align}
		B_{0} = \sigma_{B_{0}}^{2} D_{L_{0}}.
	\end{align}
	
	To produce the ensemble part $P_{f}$, we sample from a different SOAR matrix $B_{1}$ associated with a length scale of $L_{ens}$, such that,
	\begin{align}\label{B_{1}}
		B_{1} = \sigma_{P_{f}}^{2}D_{L_{ens}}.
	\end{align} To obtain the covariance matrix $P_{f}$ sampling from $B_{1}$, we generate a set of random vectors $\boldsymbol{w}_{k}\sim \mathcal{N}(0,1)$, then compute the sample covariance matrix of the set 
 $\{B_{1}^{1{/}2}\boldsymbol{w}_{1},B_{1}^{1{/}2}\boldsymbol{w}_{2},\cdots,B_{1}^{1{/}2}\boldsymbol{w}_{m}\}$. As the size of the set -- denoted by $m$ -- is restricted by $m\ < n$, $P_{f}$ is then guaranteed to be rank deficient.
	
	For the observational error covariance matrix $R$, we simply choose that, \begin{align}\label{R matrix}R = \sigma_{R_{0}}^{2}I_{p},\end{align} where $I_{p} \in \mathbb{R}^{p,p}$ is the identity matrix. 
	
	In terms of the observational operator, we consider four different types in our experiments. They are given as follows,
	\begin{align}\label{H choice}
		& {H}^{(1)}(i,j) = \begin{cases} &1, \ \ i=j,  \ \text{for} \ i=1\to p, \\ &0, \ \ i\neq j. \end{cases}; \ \ {H}^{(3)}(i,j) = \begin{cases}&1{/}5, \ \ j=\frac{n}{p}i-2\to \frac{n}{p}i+2, \ (mod(n)), \ \text{for} \ i=1\to p,  \\& 0.\end{cases};
		\nonumber \\ &{H}^{(2)}(i,j) = \begin{cases} &1, \ \ j=\frac{n}{p}i, \ \text{for} \ i=1\to p, \\ &0, \ \ else. \end{cases};  \  {H}^{(4)}(i,j) = \begin{cases}&1, \ \ i,j, \ \text{are choosen randomly and non-repeating,} \\& 0.\end{cases}
	\end{align} 
	The choice of $H^{(1)}$ corresponds to observing the first $p$ points of the domain, $H^{(2)}$ observes every $n/p$ points, $H^{(3)}$ is an observation which is a weighted sum over 5 grid points and $H^{(4)}$ chooses $p$ observations at random locations on the grid.	We note that $H^{(1)},H^{(2)}$ and $H^{(4)}$ satisfy the condition of Corollary \ref{coro 2}, while $H^{(3)}$ does not. 
	In all of the experiments, we fix the value of $n$ to be 500 such that the computation is small scale and similar to previous studies\cite{tabeart2018conditioning,haben2011conditioning}, and therefore the results can be compared. 
 
	\subsection{A comparison of different bounds}
	We note that four different versions of the bound have been presented in this paper, i.e., Theorem \ref{thm jemima 1} and \ref{thm unpreconditioned main} for the unpreconditioned cases, Theorem \ref{thm preconditioned main} and \ref{thm preconditioned not main} for the preconditioned cases with CVT. Before further investigating the details of these results, we want to determine which versions of these are the most effective (in terms of revealing the trend of the condition number itself and estimating it with a small error).
	{Considering the unpreconditioned cases, Theorem \ref{thm unpreconditioned main} has an apparent advantage of separating the impact of $P_{f}$. However, through repeated experiments we observe that in most cases the lower bound in Theorem \ref{thm unpreconditioned main} remains close to 1; hence it not effective.  On the other hand, the lower bound provided by Theorem \ref{thm jemima 1} produces a better result. Still, it is not close enough to the actual condition number to be a good estimate. 
		
{ In Figure~\ref{fig:unprecon} we show typical examples of the variation of these bounds with $\beta$.}  Figure \ref{fig lower 1} shows that, for the unpreconditioned case, as $\beta$ increases towards 1, the upper bound captures the shape of $\kappa({S_{3D}})$. {The divergence of $\kappa({S_{3D}})$ to infinity is captured by both versions of the upper bounds.} The upper bounds computed from Theorem \ref{thm ineq sum} and \ref{thm jemima 1} do not differ significantly. This is a general result for all choices of parameters we have used to conduct numerical case studies. Based on these observations, we decide to focus only on the upper bound from here onward.}
	
	In the preconditioned case, Figure \ref{fig lower 2}, we { note that the condition number is drastically reduced by the CVT compared to the unpreconditioned case. We} observe an inflection point in the upper bound of Theorem \ref{thm preconditioned main}. This coincides with the transition point of the condition number (from decreasing to increasing with $\beta$) { and so predicts the minimum of  $\kappa({S_{3D}})$.} The bound produced by Theorem \ref{thm preconditioned not main} does not capture such behaviour, while not being significantly closer to the actual condition number either.
 
	\begin{figure}[hbt!]
		\centering
		\begin{subfigure}[t]{.44\textwidth}
			\includegraphics[width=0.8\textwidth,height=0.5\textwidth]{"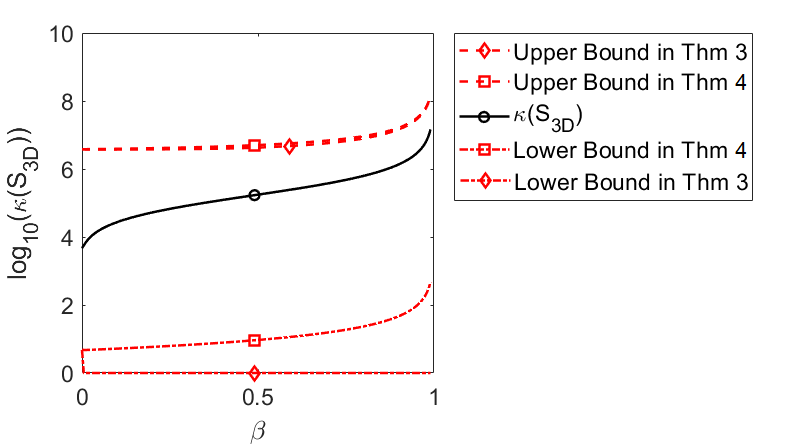"}
			\caption{}\label{fig lower 1}
		\end{subfigure}\hspace{0.6cm}
		\begin{subfigure}[t]{.44\textwidth}
			\includegraphics[width=0.8\textwidth,height=0.5\textwidth]{"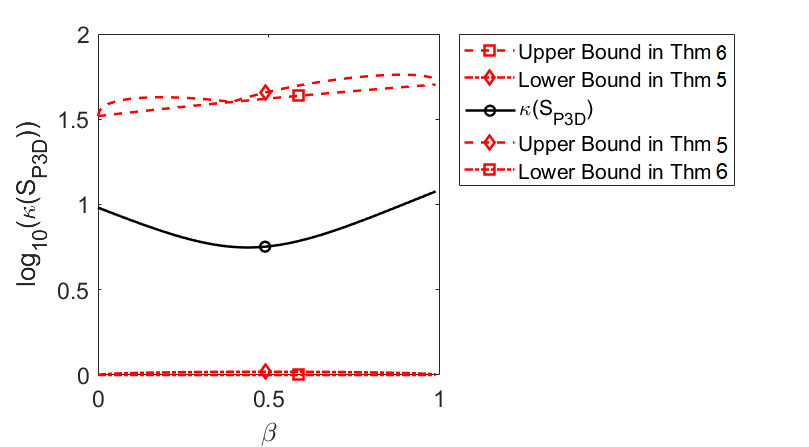"}
			\caption{}\label{fig lower 2}
		\end{subfigure}
      \captionsetup{width=.9\textwidth}

		\caption{A comparison of the variation of the bounds with $\beta$ { for (a) $\kappa({S_{3D}})$ and (b) $\kappa({S_{P3D}})$ }. The parameters are as follows, $L_{ens} = 0.05, L_{0}=0.2, \sigma_{B_{0}} = \sigma_{P_{f}} = \sigma_{R_{0}} = 1$, $m = 50, p = 100$. The linear observation operator is chosen to be $H_{0} = H^{(4)}$.} \label{fig:unprecon}
\end{figure}

	These results are valid for all different choices of matrices and parameters.Therefore we focus our study on the bound produced by Theorem \ref{thm unpreconditioned main} and Theorem \ref{thm preconditioned main} only, because of their superiority in separating the impact of different matrices and capturing the shape of the actual condition number (see Figure \ref{fig lower 2}). We note that the condition numbers are plotted in $\log_{10}$ scale in this paper.
	
	\section{Case Studies for Unpreconditioned 3D-Var}
	In this section, we illustrate the upper bound given by Theorem \ref{thm unpreconditioned main} with multiple case studies. In these case studies we change parameters associated with $B_{0},P_{f},\hat H, \hat R$ and observe the responses in the condition number of $S$ and its upper bound. In particular, we are especially interested in the relationship between the conditioning of the system and the weight of the ensemble part (i.e., $\beta$). {We note again that the case studies in the following sections use 3D-Var as a special case of 4D-Var to demonstrate the theories. Therefore, $\hat H$ and $\hat R$ are replaced with $H_{0}$ and $R_{0}$, respectively. Additionally, we will denote the Hessian as $S_{3D}$ and consider it as a special case of $S_{4D}$.}
	
	As the ensemble error covariance matrix is rank deficient, we can then expect that the balance of $B_{0}$ and $P_{f}$ is particularly important in the conditioning of the system. For example, when $P_{f}$ is the dominant component of the Hessian matrix $S_{3D}$, the condition number of $S_{3D}$ is likely to be large and the system ill-conditioned. In addition, as $\beta \to 1$, the Hessian matrix  $S_{3D} \to P_{f}^{-1} +  H_{0}^{T} R_{0}^{-1}  H_{0}$, the limit is clearly ill-defined as $P_{f}^{-1}$ does not exist. Meanwhile, the balance of $B_{0}$ and $P_{f}$ is also controlled by the relative sizes of physical length scales $L_{0},L_{ens}$ and the variances $\sigma_{B_{0}},\sigma_{P_{f}}$.
	
	Recalling the upper bound in Theorem \ref{thm unpreconditioned main}, it is predominantly controlled by the largest eigenvalues of $B_{0}$ and $P_{f}$. Meanwhile, the largest eigenvalues of $B_{0}$ and $P_{f}$ vary with their correlation length scales\cite{haben2011conditioning}. As Figure \ref{lambda1B with L} shows, the largest eigenvalue of $B_{0}$ increases with the correlation length scale; the general trend is similar for $P_{f}$ except for random fluctuations caused by the sampling noise. Consequently, we expect $ 		\kappa({S_{3D}})$ and its bound to change when the physical length scales ($L_{0},L_{ens}$) alter. 
	\begin{figure}[hbt!]
		\centering
		\includegraphics[width=0.35\textwidth,height=0.22\textwidth]{"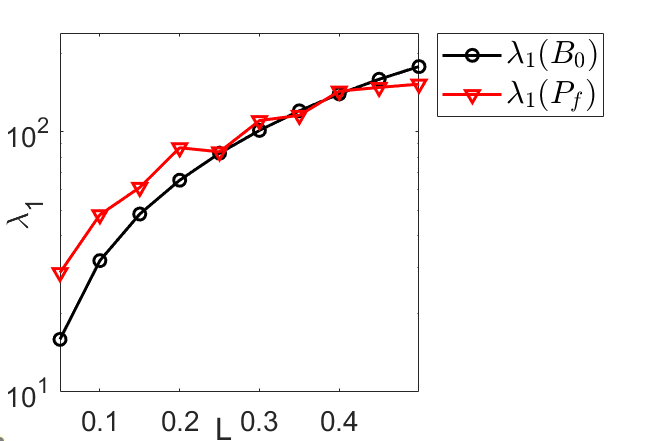"}
		\caption{The largest eigenvalues of $B_{0}$ and $P_{f}$ as functions of correlation length scale.}\label{lambda1B with L}
	\end{figure}
	
	\begin{figure}[hbt!]
		\centering
		\begin{subfigure}[t]{.44\textwidth}
			\includegraphics[width=0.87\textwidth,height=0.48\textwidth]{"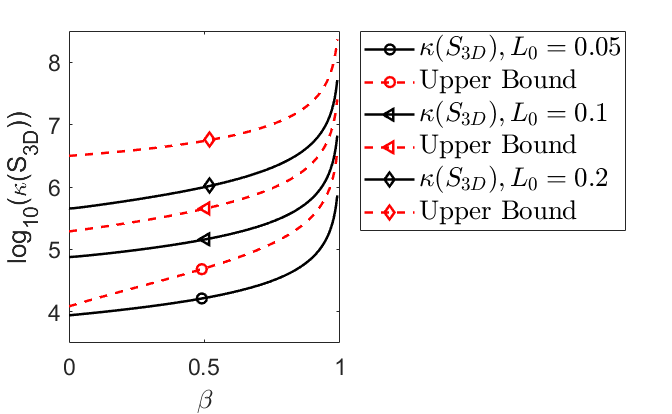"}
			\caption{}\label{L0 1}
		\end{subfigure}\hspace{0.5cm}
		\begin{subfigure}[t]{.44\textwidth}
			\includegraphics[width=0.87\textwidth,height=0.48\textwidth]{"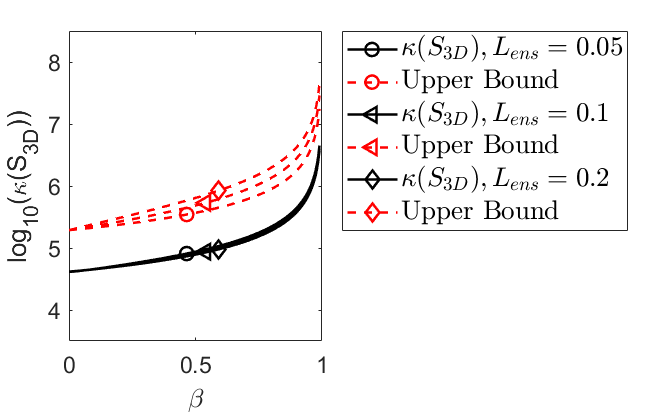"}
			\caption{ }\label{Lens 1}
		\end{subfigure}
		\vspace{0.2 cm}
		\centering
		\begin{subfigure}[t]{.44\textwidth}
			\includegraphics[width=0.85\textwidth,height=0.48\textwidth]{"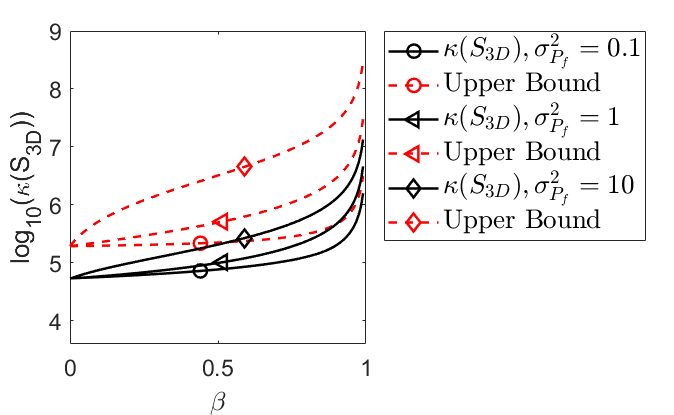"}
			\caption{}\label{sigBens}
		\end{subfigure}\hspace{0.6cm}
		\begin{subfigure}[t]{.44\textwidth}
			\includegraphics[width=0.85\textwidth,height=0.48\textwidth]{"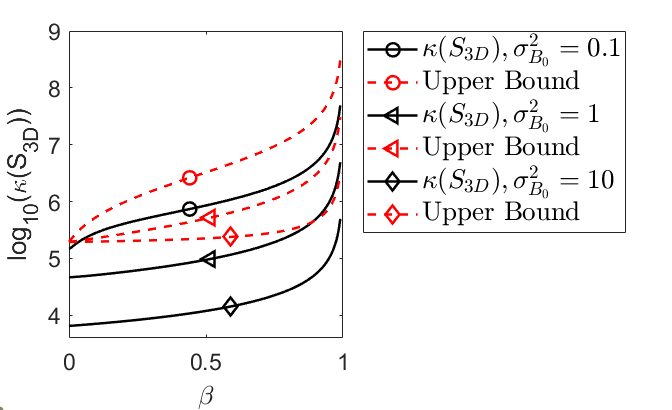"}
			\caption{ }\label{sigB0}
		\end{subfigure}	
        \captionsetup{width=.9\textwidth}

		\caption{The variation of the condition number and the upper bound with $\beta$ for different values of (a) $L_{0}$, (b) $L_{ens}$, (c) $\sigma_{B_{0}}$ and (d) $\sigma_{P_{f}}$, for the unpreconditioned case with $m = 100, p = 100$, $\sigma_{P_{f}}^{2} =\sigma_{B_{0}}^{2} = \sigma_{R_{0}}^{2} = 1$ and $H_{0}=H^{(4)}$. Figure (a) showcases the effect of $L_{0}$  (fixing $L_{ens} = 0.1$). Figure (b) demonstrates the effect of $L_{ens}$ (fixing $L_{0} = 0.1$). Figure (c) demonstrates the effect of $\sigma_{P_{f}}^{2}$ (fixing $\sigma_{B_{0}}^{2} = 1$). Figure (d) illustrates the effect of $\sigma_{B_{0}}^{2}$ (fixing $\sigma_{P_{f}}^{2} = 1$). }\label{effect of L}
	\end{figure}
	\begin{figure}[hbt!]
		\centering
		\begin{subfigure}[t]{.41\textwidth}
			\includegraphics[width=0.95\textwidth,height=0.50\textwidth]{"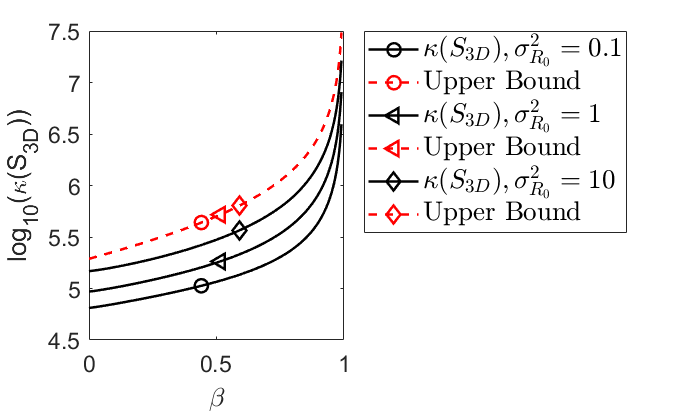"}
			\caption{}\label{sigmaR 1}
		\end{subfigure}\hspace{0.6cm}
		\begin{subfigure}[t]{.41\textwidth}
			\includegraphics[width=0.95\textwidth,height=0.50\textwidth]{"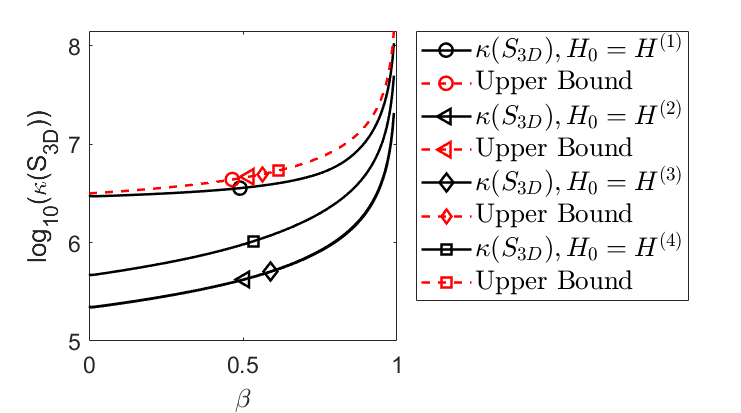"}
			\caption{}\label{H}
		\end{subfigure}
		\begin{subfigure}[t]{.41\textwidth}
			\includegraphics[width=0.95\textwidth,height=0.50\textwidth]{"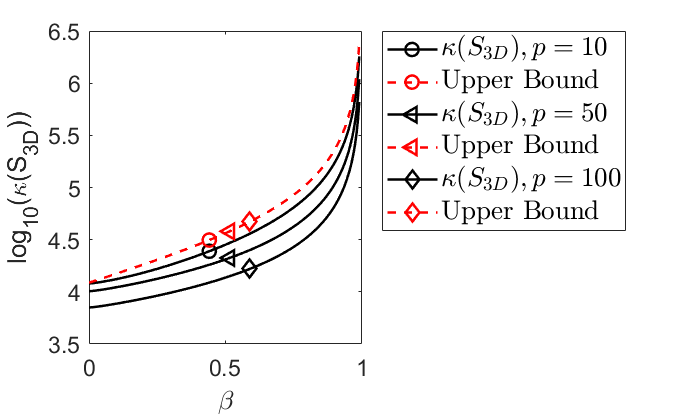"}
			\caption{}\label{p}
		\end{subfigure}
        \captionsetup{width=.9\textwidth}

		\caption{The variation of the condition number and the upper bound with $\beta$ for different (a) values of $\sigma_{R_{0}}^2$ (fixing $\sigma_{B_{0}}^{2} = \sigma_{P_{f}}^{2} =1, p = 100, \ H_{0} = H^{(4)}$),  (b) observation operators and (c) number of observations (fixing $\sigma_{R_{0}}^{2} = \sigma_{B_{0}}^{2} = \sigma_{P_{f}}^{2} = 1, H_{0} = H^{(4)}$). This group of figures showcase the unpreconditioned case with parameters of $m = 100$, $L_{ens} = L_{0} = 0.1$.}\label{effect of R}
	\end{figure}
	As an important justification of the effectiveness of the upper bound, we observe that the shape of the upper bound is similar to the condition number in all four case studies presented in Figure \ref{effect of L}. We also note that in this set of experiments, we observe that the impact of $L_{ens}$ on the conditioning and the bounds is less significant than $L_{0}$ (compare Figure \ref{L0 1},\ref{Lens 1}). In Figure \ref{sigBens}, we observe that the trend of the condition number with respect to $\sigma_{P_{f}}$ is well captured by the upper bounds. Also, in Figure \ref{sigB0}, we observe that upper bound does reveal a general trend of the condition number with respect to $\sigma_{B_{0}}$, but the upper bound does not reflect a significant reduction of the condition number at $\beta \sim 0$. The reason for this discrepancy is unknown.
	
	It is an important finding that the upper bound starts to sharply increase and diverge to infinity, and this transition occurs at about the same $\beta$ as the condition number of the Hessian. This then indicates that the bound is informative from a numerical point of view, e.g., it informs a constraint on $\beta$ if one seeks to avoid a sudden deterioration of the conditioning of the system in Hybrid 4D-Var. On the other hand, we also note that in the three cases of changing $L_{0},\sigma_{B_{0}}$ and $\sigma_{P_{f}}$, the upper bounds reflect the direction of change in the condition number of the Hessian. Thus it shows promising results that the upper bound can be used to qualitatively analyse the conditioning of the system when these parameters alter with time. Furthermore, the upper bound is two orders of magnitude above the condition number. We find similar conclusions for the cases of changing $\sigma_{R_{0}}$, the observation operator  $H_{0}$ and the number of observations $p$ (see Figure \ref{effect of R}).  
	
	We note that the upper bound for a fixed $B_{0}$ and $P_{f}$ does not change insofar as the largest eigenvalue of $H_{0}^{T}R_{0}H_{0}$ stays the same (see Theorem \ref{thm unpreconditioned main}).  Comparing $\lambda_{1}(H_{0}^{T}R_{0}H_{0})$ for different versions of $H_{0}$, we find that $H^{(1)},H^{(2)}$ and $H^{(4)}$ have the same largest eigenvalue. It is slightly smaller for $H^{(3)}$ but not significantly. We also point out that previous investigations in 4D-Var \cite{haben2011conditioning,tabeart2018conditioning} also indicate that a small change in $\sigma_{R_{0}}^{2}$ does not change the bound noticeably. Thus we anticipate that the upper bound remains similar for different versions of $H_{0}$. This is confirmed by cases in Figure \ref{H}. Meanwhile, the condition number of the Hessian is the largest for the choice of $H_{0}=H^{(1)}$, while choosing $H_{0}=H^{(2)}$ or $H_{0} = H^{(3)}$ produces a similar result and they yield the best conditioning of the Hessian; the condition number associated with the randomized version $H_{0}=H^{(4)}$ lies between those of $H^{(1)}$ and $H^{(2)},H^{(3)}$. Thus the observation indicates that evenly distributed observation points produce better conditioning of the system, while partial observations concentrated in a small region lead to the opposite.
	
	The upper bound in (\ref{ineq coro 2}) also implies that changing the number of observations does not change the upper bound (as the case study in Figure \ref{p} confirms).	Lastly, we also verified that the upper bound also remains effective when the number of ensemble members changes. We tested cases where the number of ensemble members increases from 50 to 400. In all these cases, we find that the upper bounds have similar shapes to the condition number, and they provide a good estimation of the value of the condition number.

	\section{numerical experiments for the preconditioned Hybrid 3D-Var with CVT}\label{section numerical experiments preconditioned}
	Following the case studies for the unpreconditioned cases, we now conduct similar experiments to illustrate Theorem \ref{thm preconditioned main} and examine the predictions using Theorem \ref{thm preconditioned main} for preconditioned cases with CVT. In the experiments we focus on three major issues: firstly we validate the correctness of the bound; secondly, we illustrate the bound reflects the behaviour of the conditioning in general (with a few exceptions); lastly, we make a comparison to the unpreconditioned cases. We note that each error covariance matrix and observation operator is given by Section \ref{section numerical experiments unpreconditioned}. 
	
	An immediate observation from Theorem \ref{thm preconditioned main} is that the switching point in the upper bound shifts with the relative sizes of $\lambda_{1}(B_{0})$ and $\lambda_{1}(P_{f})$. More specifically, the max function in the upper bound in (\ref{ineq preconditioned main}) switches at a larger value of $\beta$ when $\lambda_{1}(B_{0})$ increases, and the opposite occurs when $\lambda_{1}(P_{f})$ increases. Furthermore, we know that $\lambda_{1}(B_{0})$ and $\lambda_{1}(P_{f})$ become larger as $L_{0}$ and $L_{ens}$ increase (see Figure \ref{lambda1B with L}), and the same is true for increasing $\sigma_{B_{0}},\sigma_{P_{f}}$.	As Figure \ref{fig:effect of L0 Lens sigma B0 sigma Pf} shows, the results of these four case studies justifies the observations implied by Theorem \ref{thm preconditioned main}, and the switching point of the upper bound also predicts the minimum of the condition number.  	On the other hand, as $\lambda_{1}(B_{0}),\lambda_{1}(P_{f})$ increase with $L_{0},L_{ens},\sigma_{B_{0}},\sigma_{P_{f}}$, we anticipate that the upper bound would increase with these parameters. This is confirmed by the case studies (see Figure \ref{fig:effect of L0 Lens sigma B0 sigma Pf}). Crucially, we find that the upper bound predicts the trend of the condition number of Hessian with respect to $\beta$ in all four cases. For example, in Figure \ref{L0 CVT}, we observe that the inflection points of the upper bounds predict the minima of the condition numbers, and they both move rightward when $L_{0}$ increases. In Figure \ref{Lens CVT}, we find that the inflection points of the upper bounds and the minima of the condition numbers move leftwards when $L_{ens}$ increases. In Figure \ref{sigmaB0 CVT}, the inflection points of the upper bounds and the minima of the condition numbers move rightwards when $\sigma_{B_{0}}$ increases, and in figure \ref{sigmaPf CVT} we find that the trend of the condition number reacting to $\sigma_{P_{f}}$ is well captured by the upper bounds, the inflection points of the upper bounds and the minima of the condition numbers move leftwards when $\sigma_{P_{f}}$ increases. However, we want to point out that the changes in the conditioning is very limited in these preconditioned cases and are normally less than one order of magnitude.
	\begin{figure}[hbt!]
		\centering
		\begin{subfigure}[t]{.44\textwidth}
			\includegraphics[width=0.90\textwidth,height=0.55\textwidth]{"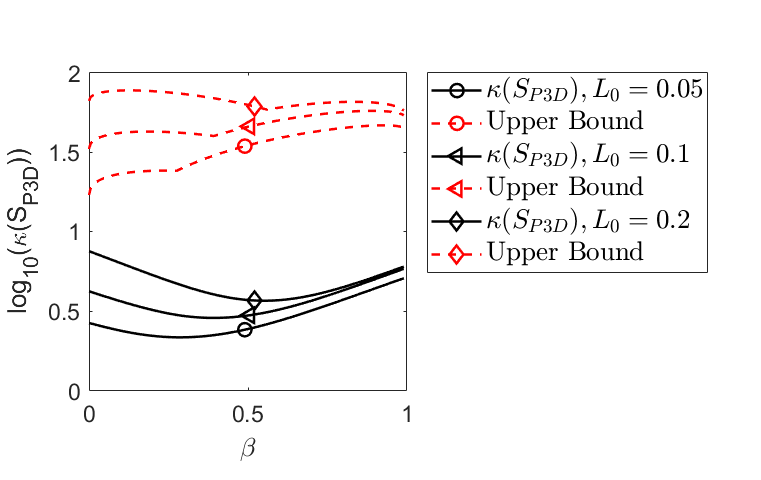"}
			\caption{}\label{L0 CVT}
		\end{subfigure}\hspace{0.6cm}
		\begin{subfigure}[t]{.44\textwidth}
			\includegraphics[width=0.90\textwidth,height=0.55\textwidth]{"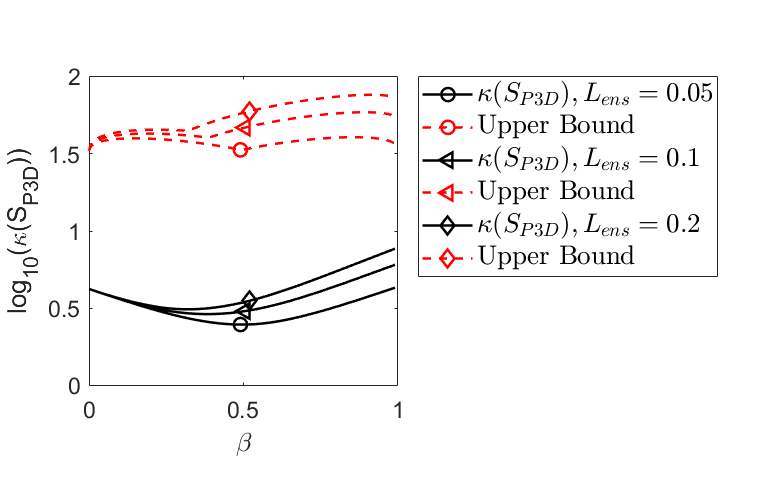"}
			\caption{}\label{Lens CVT}
		\end{subfigure}
		\vspace{0.2 cm}
		\begin{subfigure}[t]{.44\textwidth}
			\includegraphics[width=0.90\textwidth,height=0.55\textwidth]{"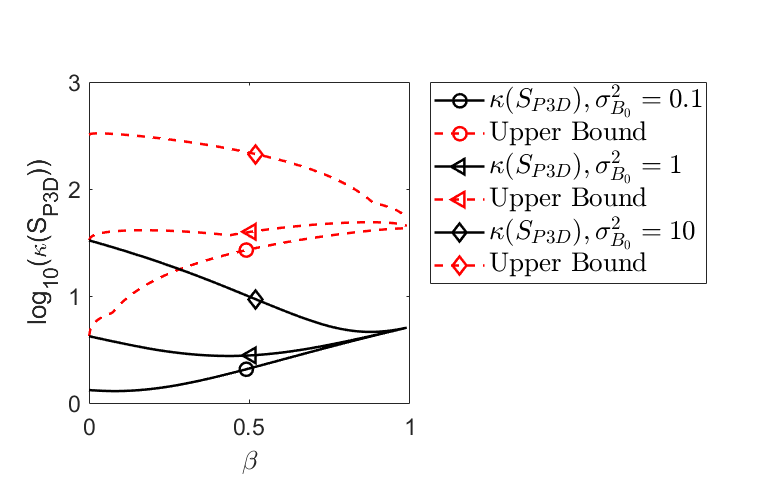"}
			\caption{}\label{sigmaB0 CVT}
		\end{subfigure}\hspace{0.6cm}
		\begin{subfigure}[t]{.44\textwidth}
			\includegraphics[width=0.90\textwidth,height=0.55\textwidth]{"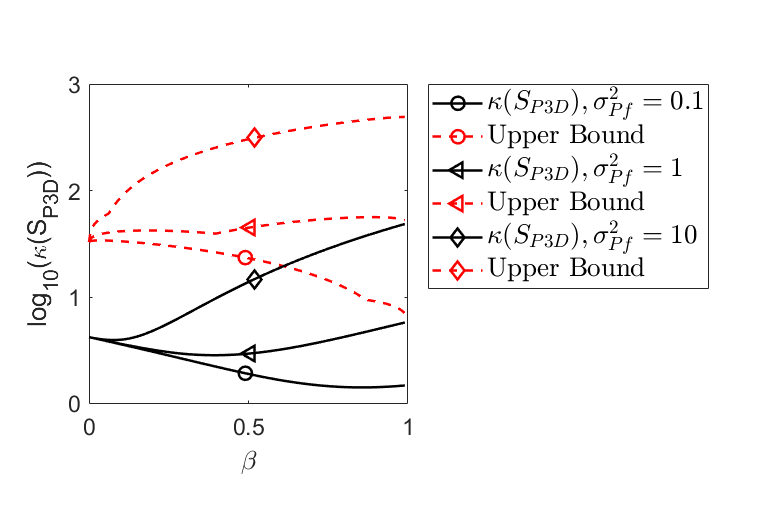"}
			\caption{}\label{sigmaPf CVT}
		\end{subfigure}\hfil
        \captionsetup{width=.9\textwidth}

		\caption{Same settings as Figure~\ref{effect of L}, but for the CVT preconditioned  cases.}
		\label{fig:effect of L0 Lens sigma B0 sigma Pf}
	\end{figure}
	
	Furthermore, compared to the unpreconditioned case, these results show a clear improvement of the conditioning in Hybrid 4D-Var with CVT (we note that this is something that is also found in standard 4D-Var in previous work\cite{tabeart2022new}). We observe that CVT leads to a maximum reduction of six magnitudes in the condition number of Hessian. We also note that in all cases presented in the preconditioned Hybrid 4D-Var, the condition number of Hessian does not diverge to infinity at $\beta = 1$ , and this is a crucial difference from the unpreconditioned Hybrid 4D-Var. 
	
	Theorem \ref{thm preconditioned main} also shows that the upper bound takes a larger value when the largest eigenvalue of $K = H_{0}^{T} R_{0} H_{0}$ is bigger. Furthermore, since $\sigma_{R_{0}}^{-2}$ is a scaling factor of $K$, we can then anticipate that the upper bound grows with a decreasing $\sigma_{R_{0}}^{2}$. The case presented in Figure \ref{preconditioned changing sigma R} not only justifies this observation but also shows that the condition number itself follows the same trend.
	
	\begin{figure}[hbt!]
		\centering
		\begin{subfigure}{.44\textwidth}
			\includegraphics[width=0.90\textwidth,height=0.55\textwidth]{"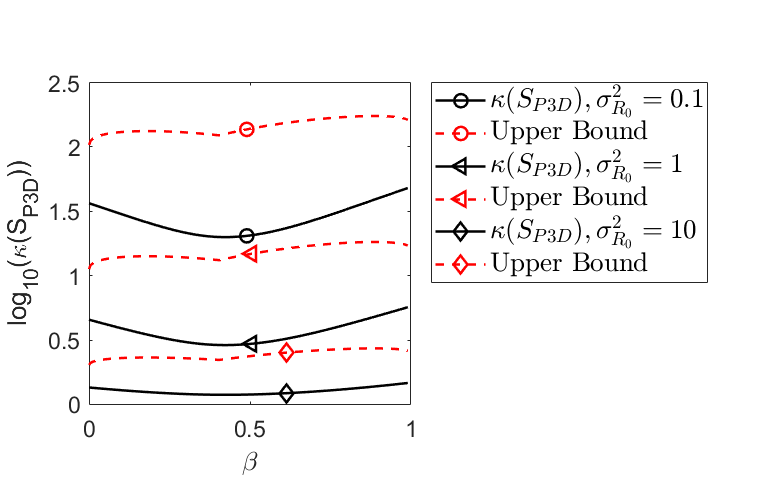"}
			\caption{}\label{preconditioned changing sigma R}
		\end{subfigure}\hspace{0.6cm}
		\begin{subfigure}{.44\textwidth}
			\includegraphics[width=0.90\textwidth,height=0.55\textwidth]{"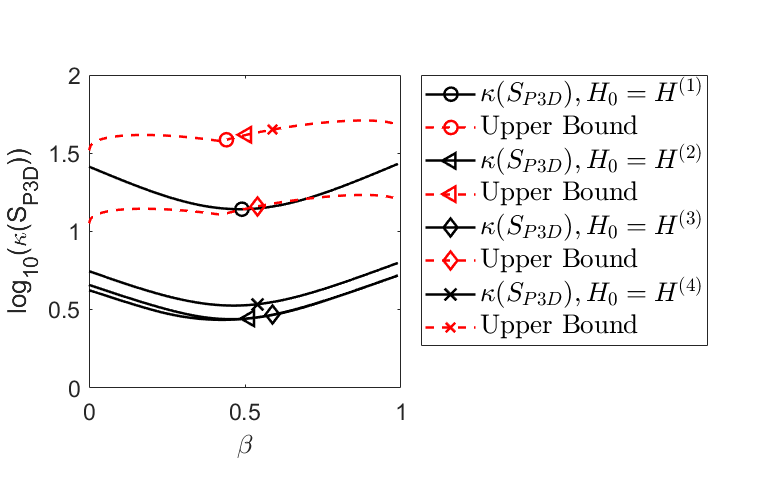"}
			\caption{}\label{cvt different H}
		\end{subfigure}
		\begin{subfigure}{.44\textwidth}
			\includegraphics[width=0.90\textwidth,height=0.55\textwidth]{"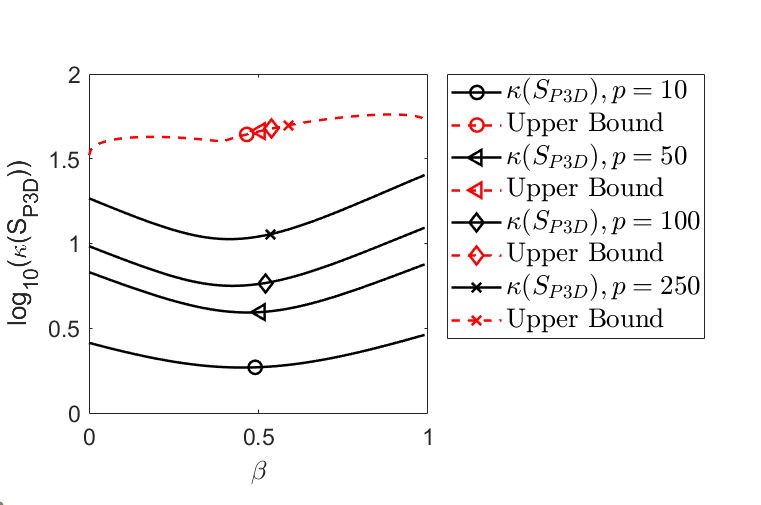"}
			\caption{}\label{cvt different p}
		\end{subfigure}
        \captionsetup{width=.9\textwidth}

		\caption{As settings as Figure~\ref{effect of R}, but for the CVT preconditioned cases.}
		\label{fig:effect of SigmaR}
	\end{figure}
	
	On the impact of choosing different versions of $H_{0}$, the trend is similar to the unpreconditioned case; we find that evenly spreading out the observation across the whole domain leads to better conditioning and skewing observations into a local region result in worse conditioning. We note that the upper bound remains the same for $H_{0} = H^{(1)},H^{(2)}$ and $H^{(4)}$ (see Figure \ref{cvt different H}, \ref{cvt different p}). This is because $K = H_{0}^{T} R_{0} H_{0}$ shares the same largest eigenvalue for these three versions, whereas $H^{(1)}$ produces a smaller maximum eigenvalue for $K$, therefore the upper bound is smaller. On the other hand, the impact of the number $p$ of observations on the condition number is opposite to that of the unpreconditioned case (this is similar to previous reports of 4D-Var \cite{tabeart2018conditioning,tabeart2022new}), but the bound estimation does not reflect this trend, as the largest eigenvalue of $K$ does not change with $p$. We note that in both unpreconditioned and preconditioned cases, the upper bounds cannot predict the impact of $p$ on the conditioning of the system. The effect of sampling noise in $P_{f}$ is similar to the unpreconditioned case, which is that the impact on the conditioning or the upper bound is insignificant. We find that this is true even with a small sample size. 
	
	As an important justification of the effectiveness of the upper bound given by Theorem \ref{thm preconditioned main}, we observe 
 the changes of the upper bound with respect to $\beta$ provide valuable information about the transition of the condition number (from decreasing to increasing); the inflection point of the upper bound predicts the minimum of the condition number. This is particularly important because it informs an optimal choice of $\beta$ from a numerical perspective of obtaining the best conditioning of the system. We therefore can conclude that the upper bound is useful for providing qualitative information about the actual conditioning of the system. 
	\section{Convergence Test of a Conjugate Gradient Routine}
	We note that there are well-known situations where the condition number provides a pessimistic indication of convergence speed (e.g. in the case of repeated or clustered eigenvalues)\cite{tabeart2022new}. { Here we use hybrid 3D-Var as a special case of hybrid 4D-Var to illustrate that for hybrid variational assimilation} the convergence speed follows a similar trend to the condition number as the weight of the ensemble part increases. 
	
	Following a similar method to section 5.3.2. of Tabeart et el\cite{tabeart2018conditioning}, we study how the speed of convergence of a conjugate
	gradient method applied to the linear system $S_{3D}\boldsymbol{x} = \boldsymbol{b}$ changes with the weight $\beta$ of the ensemble part. In the first test, the matrix $S_{3D}$ is given by the Hessian of the unpreconditioned 3D-Var (Section 5), and the vector $\boldsymbol{b}$ is given by Haben\cite{haben2011conditioning} in Section 3.2 ($\boldsymbol{b} = B^{-1}(\boldsymbol{x}_{b} - \boldsymbol{x}_{0}) - H_{0}^{T}\boldsymbol{d})$, where the vectors $\boldsymbol{x}_{b} - \boldsymbol{x}_{0},\boldsymbol{d}$ are chosen to be random at the beginning of the trial). For the computation of $S_{3D}$, we choose the parameters as follows, $L_{0} = 0.1, L_{ens} = 0.05, \sigma_{B_{0}}=\sigma_{P_{f}} = \sigma_{R_{0}} = 1$,$p = 100$ and $H_{0} = H^{(4)}$, which are in line with our previous case studies.
 
 As figure \ref{fig:noncvt cvg test 1} shows, the condition number of the Hessian shows the same trend as the number of iterations performed to reach the tolerance threshold. We can observe that the conditioning is a good proxy to study the convergence speed in this case.
	\begin{figure}[hbt!]
	\centering
	\begin{subfigure}{.40\textwidth}
		\includegraphics[width=0.97\textwidth]{"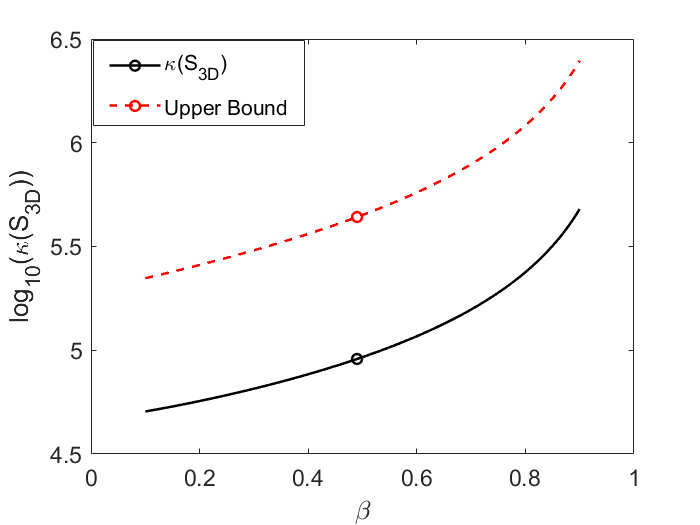"}
		\caption{}\label{noncvt cvg test kappa}
	\end{subfigure}\hspace{0.8 cm}
	\begin{subfigure}{.40\textwidth}
		\includegraphics[width=0.97\textwidth]{"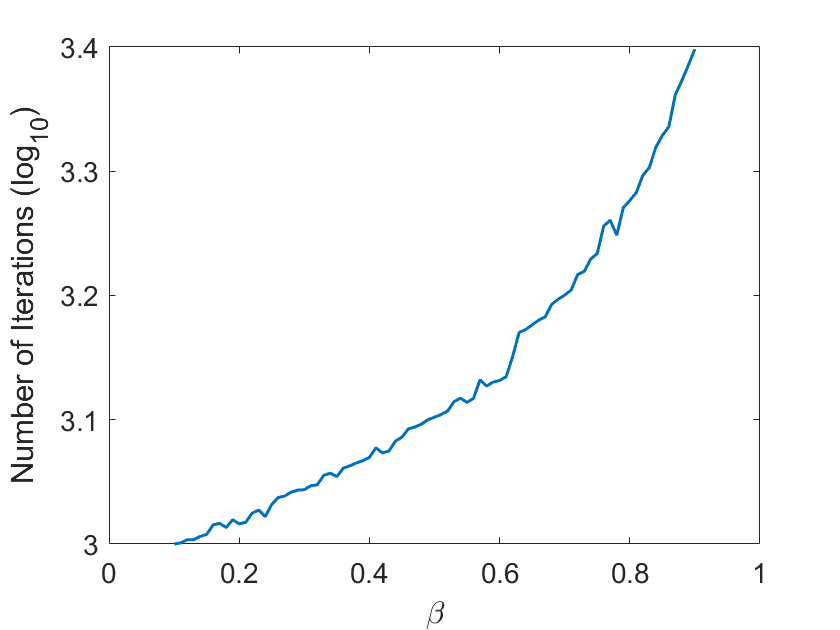"}
		\caption{}\label{noncvt cvg test num of iters}
	\end{subfigure}
       \captionsetup{width=.9\textwidth}

	\caption{Convergence test for an unpreconditioned case. Figure (a) displays the condition number of $S_{3D}$ and its upper bound. The parameters used to compute $S_{3D}$ are $L_{0} = 0.1, L_{ens} = 0.05, \sigma_{B_{0}}=\sigma_{P_{f}} = \sigma_{R_{0}} = 1$,$p = 100$ and we choose $H_{0} = H^{(4)}$. Figure (b) displays the number of iterations.}
	\label{fig:noncvt cvg test 1}
\end{figure}	
For the preconditioned case with CVT, shown in Figure~\ref{fig:cvg test 1}, we find that the difference in the condition number as $\beta$ varies is marginal relative to the unpreconditioned case. One of the reason for this is that the Hessian has more eigenvalues clustered around 1, which is a result of CVT. However, the number of iterations taken to converge still follows the general pattern of the condition number, with higher values when $\beta$ is close to 0 or 1. 
	\begin{figure}[hbt!]
	\centering
	\begin{subfigure}{.40\textwidth}
		\includegraphics[width=0.97\textwidth]{"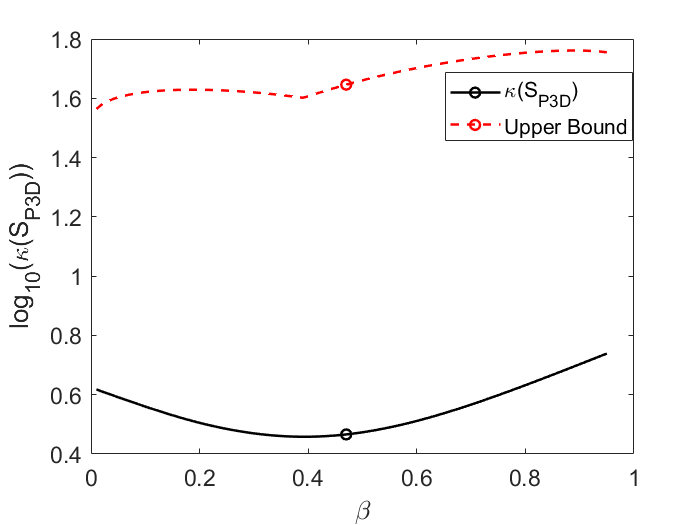"}
		\caption{}\label{cvg test kappa}
	\end{subfigure}\hspace{0.8 cm}
	\begin{subfigure}{.40\textwidth}
		\includegraphics[width=0.97\textwidth]{"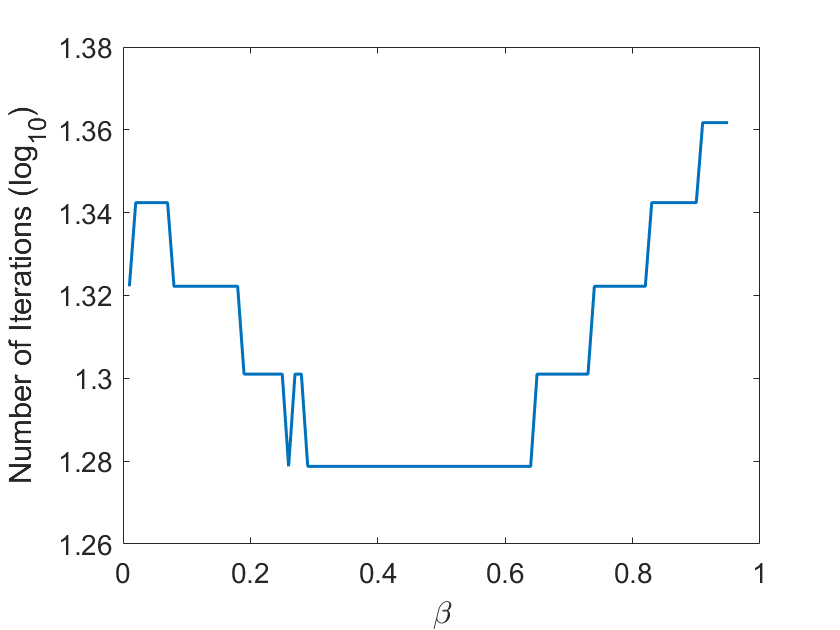"}
		\caption{}\label{cvg test num of iters}
	\end{subfigure}
       \captionsetup{width=.9\textwidth}
	\caption{Convergence test for a case preconditioned with CVT. Figure (a) displays the condition number of $S_{P3D}$ (as a special case of $S_{P4D}$) and its upper bound. {The parameters used to compute $S_{P3D}$ are $L_{0} = 0.1, L_{ens} = 0.05, \sigma_{B_{0}}=\sigma_{P_{f}} = \sigma_{R_{0}} = 1$, $p = 100$ and we choose $H_{0} = H^{(4)}$.} Figure (b) displays the number of iterations. }
	\label{fig:cvg test 1}
\end{figure}	

 { Finally, we conducted experiments with stopping criteria of different orders of magnitude to investigate the impact on the convergence speed trend. For $\epsilon>10^{-6}$ there is very little variation in the number of iterations as $\beta$ changes. For values of $\epsilon$ smaller than $10^{-6}$, the number of iterations increases slightly, but the trend in Figure 8(b) remains consistent.}
	
	\section{Summary}
	In this paper, we established a set of theories for the conditioning of Hybrid 4D-Var. These theories provide effective upper bounds for the condition number of the Hessian. These theoretical results are illustrated by numerical case studies {using the special case of Hybrid 3D-Var}. In numerical experiments we tested that the upper bounds have similar shapes to the condition number with respect to the weight of the ensemble part  (i.e. $\beta$). Thus they can provide a useful estimation of the behaviour of the condition number of the Hessian. In addition, the upper bound enabled us to study the condition number through parameters associated with different components and observe their interactions.  We conclude that the upper bound is effective in explaining the conditioning of the Hessian matrix in general. However, the lower bound does not provide useful information about the condition number in all cases studied. We summarize our findings as follows, 
	\begin{itemize}
		\item For the unpreconditioned cases, the condition number of the Hessian increases gradually at first with $\beta$ then quickly diverges to infinity as $\beta \to 1$. This transition is predicted by the upper bounds.  
		\item For the preconditioned cases with CVT, a general trend is that the condition number reduces at first and then increases as $\beta$ increases. There is an optimal $\beta$ at which the condition number is at its minimum. The upper bound predicts the optimal $\beta$ effectively. 
		\item The preconditioning with CVT improves the conditioning drastically and eliminates the divergence of the condition number of Hessian at $\beta = 1$.
		\item For  preconditioned  cases, the upper bound changes in the same direction as the condition number of the Hessian with respect to changes in $L_{0},L_{ens},\sigma_{B_{0}},\sigma_{P_{f}}$ and $\sigma_{R_{0}}$. In unpreconditioned cases, we have similar conclusion for $L_{0},\sigma_{B_{0}}$ and $\sigma_{P_{f}}$.		
		
        \end{itemize}
	
	In preconditioned cases, we find that the upper bound reveals the trend of the condition number with respect to changing parameters such as the correlation length scale and the variance. However, the theories in section 3 cannot explain the impact of the four different choices of $H_{0}$ that we presented in this paper. Furthermore, the bounds do not change with the number of observations $p$ and the sample size $m$, although $p$ does directly influence the actual condition number. 
	
	Meanwhile, the tests show that these theories do provide useful predictions on the impact of the balancing of variances and correlation length scales for the preconditioned cases. In unpreconditioned cases, the upper bound can predict well the influence of $L_{0}$ and $\sigma_{P_{f}}$. The bounds can inform the impact of these components on the convergence of iterative numerical algorithms.
	
	{It is well-known that the condition number of the Hessian matrix is a useful proxy to study the convergence speed of the least-squares minimisation of Hybrid 4D-Var. The results presented in this paper could then inform applications in terms of the restriction of the weight of the ensemble part (for the unpreconditioned cases), such that extreme ill-conditioning can be avoided. For the preconditioned Hybrid 4D-Var with CVT, we established a theory that effectively predicts an optimal weight of the ensemble part such that the conditioning is optimal.  }
	
	\section*{Acknowledgements}
	This work is funded by the UK Engineering and Physical Sciences Research Council, grant number EP/V061828/1, and in part by the NERC National Centre for Earth Observation.

 \bibliographystyle{unsrt}
 \bibliography{main} 

\begin{thebibliography}{10}

\bibitem{nichols2010mathematical}
Nancy~Kay Nichols.
\newblock Mathematical concepts of data assimilation.
\newblock In {\em Data assimilation}, pages 13--39. Springer, 2010.

\bibitem{gratton2007approximate}
Serge Gratton, Amos~S Lawless, and Nancy~K Nichols.
\newblock Approximate gauss--newton methods for nonlinear least squares
  problems.
\newblock {\em SIAM Journal on Optimization}, 18(1):106--132, 2007.

\bibitem{haben2011conditioning}
Stephen~A Haben, Amos~S Lawless, and Nancy~K Nichols.
\newblock Conditioning of incremental variational data assimilation, with
  application to the met office system.
\newblock {\em Tellus A: Dynamic Meteorology and Oceanography}, 63(4):782--792,
  2011.

\bibitem{bannister2017review}
R~N Bannister.
\newblock A review of operational methods of variational and
  ensemble-variational data assimilation.
\newblock {\em Quarterly Journal of the Royal Meteorological Society},
  143(703):607--633, 2017.

\bibitem{smith2017estimating}
Polly~J Smith, Amos~S Lawless, and Nancy~K Nichols.
\newblock Estimating forecast error covariances for strongly coupled
  atmosphere--ocean 4d-var data assimilation.
\newblock {\em Monthly Weather Review}, 145(10):4011--4035, 2017.

\bibitem{smith2018treating}
Polly~J Smith, Amos~S Lawless, and Nancy~K Nichols.
\newblock Treating sample covariances for use in strongly coupled
  atmosphere-ocean data assimilation.
\newblock {\em Geophysical Research Letters}, 45(1):445--454, 2018.

\bibitem{golub2013matrix}
Gene~H Golub and Charles~F Van~Loan.
\newblock {\em Matrix computations}.
\newblock JHU press, 2013.

\bibitem{tabeart2018conditioning}
Jemima~M Tabeart, Sarah~L Dance, Stephen~A Haben, Amos~S Lawless, Nancy~K
  Nichols, and Joanne~A Waller.
\newblock The conditioning of least-squares problems in variational data
  assimilation.
\newblock {\em Numerical Linear Algebra with Applications}, 25(5):e2165, 2018.

\bibitem{tabeart2022new}
Jemima~M Tabeart, Sarah~L Dance, Amos~S Lawless, Nancy~K Nichols, and Joanne~A
  Waller.
\newblock New bounds on the condition number of the hessian of the
  preconditioned variational data assimilation problem.
\newblock {\em Numerical Linear Algebra with Applications}, 29(1):e2405, 2022.

\bibitem{lawless2005investigation}
A~S Lawless, S~Gratton, and N~K Nichols.
\newblock An investigation of incremental 4d-var using non-tangent linear
  models.
\newblock {\em Quarterly Journal of the Royal Meteorological Society: A journal
  of the atmospheric sciences, applied meteorology and physical oceanography},
  131(606):459--476, 2005.

\bibitem{mercier2019speeding}
Fran{\c{c}}ois Mercier, Yann Michel, Thibaut Montmerle, Pierre Jolivet, and
  Selime G{\"u}rol.
\newblock Speeding up the ensemble data assimilation system of the limited-area
  model of m{\'e}t{\'e}o-france using a block krylov algorithm.
\newblock {\em Quarterly Journal of the Royal Meteorological Society},
  145(720):910--929, 2019.

\bibitem{desroziers2012accelerating}
G{\'e}rald Desroziers and Lo{\"\i}k Berre.
\newblock Accelerating and parallelizing minimizations in ensemble and
  deterministic variational assimilations.
\newblock {\em Quarterly Journal of the Royal Meteorological Society},
  138(667):1599--1610, 2012.

\bibitem{hatfield2020single}
Sam Hatfield, Andrew McRae, Tim Palmer, and Peter D{\"u}ben.
\newblock Single-precision in the tangent-linear and adjoint models of
  incremental 4d-var.
\newblock {\em Monthly Weather Review}, 148(4):1541--1552, 2020.

\bibitem{aabaribaoune2020estimation}
Mohammad~El Aabaribaoune, Emanuele Emili, and Vincent Guidard.
\newblock Estimation of the error covariance matrix for iasi radiances and its
  impact on ozone analyses.
\newblock {\em Atmospheric Measurement Techniques Discussions}, 2020:1--26,
  2020.

\bibitem{buehner2005ensemble}
Mark Buehner.
\newblock Ensemble-derived stationary and flow-dependent background-error
  covariances: Evaluation in a quasi-operational nwp setting.
\newblock {\em Quarterly Journal of the Royal Meteorological Society: A journal
  of the atmospheric sciences, applied meteorology and physical oceanography},
  131(607):1013--1043, 2005.

\bibitem{wilkinson1971algebraic}
JH~Wilkinson.
\newblock The algebraic eigenvalue problem.
\newblock In {\em Handbook for Automatic Computation, Volume II, Linear
  Algebra}. Springer-Verlag New York, 1971.

\bibitem{weyl1912asymptotische}
Hermann Weyl.
\newblock Das asymptotische verteilungsgesetz der eigenwerte linearer
  partieller differentialgleichungen (mit einer anwendung auf die theorie der
  hohlraumstrahlung).
\newblock {\em Mathematische Annalen}, 71(4):441--479, 1912.

\bibitem{wang1992some}
Boying Wang and Fuzhen Zhang.
\newblock Some inequalities for the eigenvalues of the product of positive
  semidefinite hermitian matrices.
\newblock {\em Linear algebra and its applications}, 160:113--118, 1992.

\bibitem{simonin2014doppler}
D~Simonin, SP~Ballard, and Z~Li.
\newblock Doppler radar radial wind assimilation using an hourly cycling 3d-var
  with a 1.5 km resolution version of the met office unified model for
  nowcasting.
\newblock {\em Quarterly Journal of the Royal Meteorological Society},
  140(684):2298--2314, 2014.

\bibitem{gong2020inverse}
Helin Gong, Yingrui Yu, Qing Li, and Chaoyu Quan.
\newblock An inverse-distance-based fitting term for 3d-var data assimilation
  in nuclear core simulation.
\newblock {\em Annals of Nuclear Energy}, 141:107346, 2020.

\bibitem{cheng2021observation}
Sibo Cheng and Mingming Qiu.
\newblock Observation error covariance specification in dynamical systems for
  data assimilation using recurrent neural networks.
\newblock {\em Neural Computing and Applications}, pages 1--19, 2021.

\bibitem{waller2016theoretical}
Joanne~A Waller, Sarah~L Dance, and Nancy~K Nichols.
\newblock Theoretical insight into diagnosing observation error correlations
  using observation-minus-background and observation-minus-analysis statistics.
\newblock {\em Quarterly Journal of the Royal Meteorological Society},
  142(694):418--431, 2016.

\end{thebibliography}

\end{document}